\documentclass{amsproc}
\usepackage{amssymb}
\usepackage{amsmath}
\usepackage{amsfonts}
\usepackage[all]{xy}
\usepackage{amscd}

\theoremstyle{plain}

\newtheorem{conclusion}{Conclusion}

\newtheorem{corollary}{Corollary}

\newtheorem{definition}{Definition}

\newtheorem{remark}{Remark}

\newtheorem{theorem}{Theorem}
\numberwithin{equation}{section}
\newcommand{\squarecell}[9]{
\xymatrix@R=1pc@C=5pc{
#1 \ar@{->}[r]|{#2}="2" \ar[d]_{#4} & #3 \ar[d]^{#6} \\
#7 \ar@{->}[r]|{#8}="8" & #9  \\
\ar@{=>}^{#5} "2";"8"}}
\newcommand{\quadrado}[8]{
\xymatrix{
#1 \ar[r]^{#2} \ar[d]_{#4}  & #3  \ar[d]^{#5} \\
#6  \ar[r]^{#7}  &  #8 }}
\newcommand{\cela}[5]{
\xymatrix{
#1  \ar@/^1pc/[r]^{#2}="2"  \ar@/_1pc/[r]_{#4 }="4"  &  #5  \\
\ar@{=>}|{#3 } "2";"4"}}
\newcommand{\rg}[5]{    \xymatrix{#1 \ar@<1ex>[r]^{#2} \ar@<-1ex>[r]_{#4} & #5  \ar[l]|{#3} } }
\newcommand{\splitepi}[4]{    \xymatrix{#1 \ar@<.5ex>[r]^{#2} & #4  \ar@<.5ex>[l]^{#3} } }
\newcommand{\splitext}[6]{    \xymatrix{#1 \ar[r]^{#2} & #3 \ar@<.5ex>[r]^{#4} & #6  \ar@<.5ex>[l]^{#5} } }

\begin{document}
\title{Internal precategories relative to split epimorphisms}
\author{N. Martins-Ferreira}
\address{Polytechnic Institute of Leiria - CDRSP}
\email{nelsonmf@estg.ipleiria.pt}
\urladdr{http://www.estg.ipleiria.pt/\symbol{126}nelsonmf}
\date{25Aug2008.}
\subjclass[2000]{Primary 8D35; Secondary 18E05.}
\keywords{Internal precategory, internal reflexive graph, internal action,
half-reflection, crossed-module, precrossed-module, additive, semi-additive,
binary coproducts, kernels of split epimorphisms, split short five lemma.}
\thanks{The author thanks to Professors G. Janelidze and D. Bourn for much
appreciated help of various kinds.}

\begin{abstract}
For a given category B we are interested in studying internal categorical
structures in B. This work is the starting point, where we consider
reflexive graphs and precategories (i.e., for the purpose of this note, a
simplicial object truncated at level 2). We introduce the notions of
reflexive graph and precategory relative to split epimorphisms. We study the
additive case, where the split epimorphisms are \textquotedblleft coproduct
projections", and the semi-additive case where split epimorphisms are
\textquotedblleft semi-direct product projections". The result is a
generalization of the well known equivalence between precategories and
2-chain complexes. We also consider an abstract setting, containing, for
example, strongly unital categories.
\end{abstract}

\maketitle

\section{Introduction}

A internal reflexive graph in the category Ab, of abelian groups, is
completely determined, up to an isomorphism, by a morphism $%
h:X\longrightarrow B$ and it is of the following form%
\begin{equation*}
\xymatrix{
X \oplus B
\ar@<1ex>[r]^{\pi_2}
\ar@<-1ex>[r]_{[h \, 1]}
& B
\ar[l]|(.4){\iota_2}
}%
.
\end{equation*}%
An internal precategory (i.e., for the purpose of this work, a simplicial
object truncated at level 2) is, in the first place, determined by a diagram%
\begin{equation*}
\xymatrix{
Y
\ar@<1ex>[r]^{a}
\ar@<-1ex>[r]_{u}
& X
\ar[l]|{b}
\ar[r]^{h}
& B
}%
\end{equation*}%
such that 
\begin{equation*}
ab=1=ub\ ,\ ha=hu,
\end{equation*}%
and later, with a further analysis, it simplifies to a 2-chain (see \cite{BR}
and \cite{DBourn} for more general results on this topic), i.e.,%
\begin{equation*}
Z\overset{t}{\longrightarrow }X\overset{h}{\longrightarrow }B\ \ ,\ \ ht=0.
\end{equation*}%
And it is always of the following form%
\begin{equation*}
\xymatrix{
(Z \oplus X) \oplus (X \oplus B)
\ar@<4ex>[rr]^{\pi_2}
\ar@<-4ex>[rr]_{\pi_2 \oplus [h \, 1]}
\ar[rr]|(0.6){[\iota_1[t \, 1], \, 1]}
& &X \oplus B
\ar@<-2ex>[ll]|{\iota_2}
\ar@<2ex>[ll]|(.4){\iota_2 \oplus \iota_2}
\ar@<1ex>[r]^{\pi_2}
\ar@<-1ex>[r]_{[h \, 1]}
& B
\ar[l]|(.4){\iota_2}
}%
.
\end{equation*}

The same result holds for arbitrary additive categories with kernels. In
this work we will be interested in answering the following question:
\textquotedblleft what is the more general setting where one can still have
similar results?".

An old observation of G. Janelidze, says that \textquotedblleft since every
higher dimensional categorical structure is obtained from an n-simplicial
object; and since a simplicial object is build up from split epis; and,
since in Ab, every split epi is simply a biproduct projection, then it is
expected that, when internal to Ab, all the higher dimensional structures
reduce to categories of presheaves". We use this observation as a motivation
for the study of internal categorical structures restricted to a given
subclass of split epis.

In particular in this work we will be interested in the study of the notion
of internal reflexive graph (1-simplicial object) and internal precategory
(2-simplicial object) relative to a given subclass of split epis, such as
for example coproduct projections (in a pointed category with coproducts),
or product projections, or semidirect product projections, or etc.

In some cases the given subclass is saturated (in the words of D. Bourn), as
it happens for example in an additive category with kernels for the subclass
of biproduct projections. However, in general, this is not the case;
nevertheless, in some cases, interesting notions do occur.

Take for example the category of pointed sets and the class of coproduct
projections, that is, consider only split epis of the form%
\begin{equation*}
\xymatrix{
X \sqcup B
\ar@<1ex>[r]^{[0 \, 1]}
& B
\ar[l]|{\iota_2}
}%
;
\end{equation*}%
It follows that a reflexive graph relative to coproduct projections is
completely determined by a morphism%
\begin{equation*}
h:X\longrightarrow B
\end{equation*}%
and it is of the form%
\begin{equation*}
\xymatrix{
X \sqcup B
\ar@<1ex>[r]^{[0 \, 1]}
\ar@<-1ex>[r]_{[h \, 1]}
& B
\ar[l]|{\iota_2}
}%
;
\end{equation*}%
While a precategory is determined by a diagram%
\begin{equation*}
\xymatrix{
Y
\ar@<1ex>[r]^{a}
\ar@<-1ex>[r]_{u}
& X
\ar[l]|{b}
\ar[r]^{h}
& B
}%
\ \ \ ,\ \ \ ab=1=ub\ ,\ ha=hu,
\end{equation*}%
and it is of the form%
\begin{equation*}
\xymatrix{
Y \sqcup (X \sqcup B)
\ar@<4ex>[rr]^{[0 \, 1]}
\ar@<-4ex>[rr]_{a \sqcup [h \, 1]}
\ar[rr]|(0.6){[\iota_1 u, \, 1]}
& &X \sqcup B
\ar@<-2ex>[ll]|{\iota_2}
\ar@<2ex>[ll]|{b \sqcup \iota_2}
\ar@<1ex>[r]^{[0 \, 1]}
\ar@<-1ex>[r]_{[h \, 1]}
& B
\ar[l]|{\iota_2}
}%
\end{equation*}%
where the key factor for this result to hold is the fact that $\iota _{1}$
is the kernel of $[0\ 1]$.

Furthermore, every such reflexive graph $\left( h:X\longrightarrow B\right) $
may be considered as a precategory,%
\begin{equation*}
\xymatrix{
X
\ar@<1ex>[r]^{1}
\ar@<-1ex>[r]_{1}
& X
\ar[l]|{1}
\ar[r]^{h}
& B
}%
,
\end{equation*}%
and it is a internal category if the kernel of $h$ is trivial, which is the
same as saying that the following square%
\begin{equation*}
\quadrado
{  X \sqcup (X \sqcup B) }       {  [0 \, 1] }     {  X \sqcup B   }
{ 1 \sqcup [h \, 1]  }                  {   [h \, 1]  }
{ X \sqcup B  }       {  [0 \, 1] }     {  B   }%
\end{equation*}%
is a pullback.

Specifically, given a morphism $h:X\longrightarrow B$ with trivial kernel
(in pointed sets), the internal category it describes is the following: the
objects are the elements of $B$; the morphisms are the identities $1_{b}$
for each $b\in B$ and also the elements $x\in X$, except for the
distinguished element $0\in X$ that is identified with $0\in B$. The domain
of every $x$ in $X$ is $0\in B$ and the codomain is $h\left( x\right) $.
Since all arrows (except identities) start from $0\in B$, and because the
kernel of $h$ is trivial, two morphisms $x$ and $x^{\prime }$ (other than $0$%
) never compose. The picture is a star with all arrows from the origin with
no nontrivial loops.

On the other hand, again in pointed sets, if considering the subclass of
split epis that are product projections, that is, of the form%
\begin{equation*}
\xymatrix{
X \times B
\ar@<1ex>[r]^{\pi_2}
& B
\ar[l]^(0.4){<0,1>}
}%
;
\end{equation*}%
then the following result is obtained.

A internal reflexive graph relative to product projections is given by a map%
\begin{equation*}
\xi :X\times B\longrightarrow B;\ \ \ \left( x,b\right) \mapsto x\cdot b
\end{equation*}%
such that $0\cdot b=b$ for all $b\in B$; and it is of the form%
\begin{equation*}
\xymatrix@C=3pc{
X \times B
\ar@<1ex>[r]^{\pi_2}
\ar@<-1ex>[r]_{\xi}
& B
\ar[l]|(0.4){<0,1>}
}%
.
\end{equation*}%
A internal precategory, relative to product projections, is given by%
\begin{equation*}
Y\times \left( X\times B\right) \overset{\mu }{\longrightarrow }X\times B%
\overset{\xi }{\longrightarrow }B\ \ \ ,\ \ \ Y\overset{\alpha }{%
\overrightarrow{\underset{\beta }{\longleftarrow }}}X
\end{equation*}%
such that 
\begin{eqnarray*}
\alpha \beta &=&1, \\
\mu \left( y,x,b\right) &=&\left( y+_{b}x,\ b\right) , \\
0+_{b}x &=&x=\beta \left( x\right) +_{b}0 \\
\left( y+_{b}x\right) \cdot b &=&\alpha \left( y\right) \cdot \left( x\cdot
b\right) ;
\end{eqnarray*}%
and it is of the form%
\begin{equation*}
\xymatrix@C=4pc{
Y \times (X \times B)
\ar@<4ex>[r]^{\pi_2}
\ar@<-4ex>[r]_{\alpha \times \xi}
\ar[r]|{\mu}
& X \times B
\ar@<-2ex>[l]|(0.4){<0,1>}
\ar@<2ex>[l]|(.4){\beta \times <0,1>}
\ar@<1ex>[r]^{\pi_2}
\ar@<-1ex>[r]_{\xi}
& B
\ar[l]|{<0,1>}
}%
.
\end{equation*}%
In particular if $X=Y$ and $\alpha =\beta =1$, we obtain a internal category
(not necessarily associative), because the square%
\begin{equation*}
\quadrado
{ X\times (X\times B)  }       { \pi_2  }     { X\times B    }
{ 1\times \xi  }                  {  \xi   }
{ X\times B  }       {  \pi_2 }     {  B   }%
\end{equation*}%
is a pullback.

This means that a internal category in pointed sets and relative to product
projections is given by two maps%
\begin{eqnarray*}
\mu &:&X\times \left( X\times B\right) \longrightarrow X\times B;\ \ \left(
x,x^{\prime },b\right) \mapsto \left( x+_{b}x^{\prime },b\right) \\
\xi &:&X\times B\longrightarrow B;\ \ \ \left( x,b\right) \mapsto x\cdot b
\end{eqnarray*}%
such that%
\begin{eqnarray*}
0+_{b}x &=&x=x+_{b}0 \\
0\cdot b &=&b \\
\left( x+_{b}x^{\prime }\right) \cdot b &=&x\cdot \left( x^{\prime }\cdot
b\right)
\end{eqnarray*}%
and in order to have associativity one must also require the additional
condition%
\begin{equation*}
\left( x^{\prime \prime }+_{(x\cdot b)}x^{\prime }\right) +_{b}x=x^{\prime
\prime }+_{b}\left( x^{\prime }+_{b}x\right) .
\end{equation*}%
Specifically, given a structure as above in pointed sets, the corresponding
internal category that it represents is the following. The objects are the
elements of $B$. The morphisms are pairs $\left( x,b\right) $ with domain $b$
and codomain $x\cdot b$. The composition of%
\begin{equation*}
b\overset{(x,b)}{\longrightarrow }x\cdot b\overset{(x^{\prime },b^{\prime })}%
{\longrightarrow }x^{\prime }\cdot \left( x\cdot b\right)
\end{equation*}%
is the pair $\left( x^{\prime }+_{b}x,b\right) $.

We will observe that for a given subclass of split epis, when the following
two properties are present for every split epi $\left( A,\alpha ,\beta
,B\right) $ in the subclass:\newline
(a) the morphism $\alpha :A\longrightarrow B$ has a kernel, say $%
k:X\longrightarrow A$\newline
(b) the pair $\left( k,\beta \right) $ is jointly epic\newline
then a reflexive graph relative to the given subclass is determined by a
split epi in the subclass, say $\left( A,\alpha ,\beta ,B\right) $ together
with a \emph{central morphism}%
\begin{equation*}
h:X\longrightarrow B
\end{equation*}%
where a central morphism (see \cite{Bourn&Gran}) is such that there is a
(necessarily unique) morphism, denoted by $[h\ 1]:A\longrightarrow B$ with
the property%
\begin{equation*}
\lbrack h\ 1]\beta =1\ \ ,\ \ [h\ 1]k=h,
\end{equation*}%
where $k:X\longrightarrow A$ is the kernel of $\alpha $.

In the case of Groups, considering the subclass of split epis given by
semi-direct product projections%
\begin{equation*}
\xymatrix{
X \rtimes B
\ar@<1ex>[r]^{\pi_2}
& B
\ar[l]^(0.4){<0,1>}
}%
,
\end{equation*}%
the notion of central morphism $h:X\longrightarrow B$ (together with a
semidirect product projection, or an internal group action) corresponds to
the usual definition of pre-crossed module.

This fact may lead us to consider an abstract notion of semidirect product
as a diagram in a category satisfying some universal property.

In \cite{Berndt} O. Berndt proposes the categorical definition of semidirect
products as follows: the semidirect product of $X$ and $B$ (in a pointed
category) is a diagram%
\begin{equation*}
X\overset{k}{\longrightarrow }A\overset{\alpha }{\overrightarrow{\underset{%
\beta }{\longleftarrow }}}B
\end{equation*}%
such that $\alpha \beta =1$ and $k=\ker \alpha $.

We now see that it would be more reasonable to adjust this definition as
follows: in a pointed category, the semidirect product of $X$ and $B$,
denoted $X\rtimes B$, is defined together with two morphisms%
\begin{equation*}
X\overset{k}{\longrightarrow }X\rtimes B\overset{\beta }{\longleftarrow }B
\end{equation*}%
satisfying the following three conditions:\newline
(a) the pair $\left( k,\beta \right) $ is jointly epic\newline
(b) the zero morphism%
\begin{equation*}
0:X\longrightarrow B
\end{equation*}%
is central, that is, there exists a (necessarily unique) morphism $[0\
1]:X\rtimes B\longrightarrow B$ with $[0\ 1]\beta =1\ \ $and$\ \ [0\ 1]k=0$%
\newline
(c) $k$ is the kernel of $[0\ 1]$.

We must add that this object $X\rtimes B$ may not be uniquely determined
(even up to isomorphism), to achieve that we simply require the pair $\left(
k,\beta \right) $ to be universal with the above properties.

We also remark that we have not investigate further the consequences of such
a definition. It will only be done in some future work. We choose to mention
it at this point because it is related with the present work.

Another example, of considering internal categories relative to split epis,
may be found in \cite{Patch} where A. Patchkoria shows that, in the category
of Monoids, the notion of internal category relative to semidirect product
projections is in fact equivalent to the notion of a Schreier category.

This work is organized as follows.

First we recall some basic definitions, and introduce a concept that is
obtained by weakening the notion of reflection, so that we choose to call it
half-reflection.

Next we study the case of additivity, and find minimal conditions on a
category $\mathbf{B}$ in order to have%
\begin{eqnarray*}
RG\left( \mathbf{B}\right) &\sim &Mor\left( \mathbf{B}\right) \\
PC\left( \mathbf{B}\right) &\sim &2\text{-}Chains\left( \mathbf{B}\right)
\end{eqnarray*}%
the usual equivalences between reflexive graphs and morphisms in $\mathbf{B}$%
, precategories and 2-chains in $\mathbf{B}$. We show that this is the case
exactly when $\mathbf{B}$ is pointed (but not necessarily with a zero
object), has binary coproducts and kernels of split epis, and satisfies the
following two conditions (see Theorem \ref{Th 2}):\newline
(a) $\iota _{1}$ is the kernel of $[01]$%
\begin{equation*}
X\overset{\iota _{1}}{\longrightarrow }X\sqcup B\overset{[0\ 1]}{%
\longrightarrow }B
\end{equation*}%
$\newline
$(b) the split short five lemma holds.

Later we investigate the same notions, and essentially obtain the same
results, for the case of semi-additivity, by replacing coproduct projections
by semidirect product projections, where the notion of semi-direct product
is associated with the notion of internal actions in the sense of \cite{GJ1}
and \cite{JB}.

At the end we describe the same situation for a general setting, specially
designed to mimic internal actions and semidirect products. An application
of the results is given for the category of unitary magmas with right
cancellation.

\section{Definitions}

\begin{definition}[reflection]
A functor $I:\mathbf{A}\longrightarrow \mathbf{B}$ is a reflection when
there is a functor%
\begin{equation*}
H:\mathbf{B}\longrightarrow \mathbf{A}
\end{equation*}%
and a natural transformation%
\begin{equation*}
\rho :1_{\mathbf{A}}\longrightarrow HI
\end{equation*}%
satisfying the following conditions%
\begin{eqnarray*}
IH &=&1_{\mathbf{B}} \\
I\circ \rho &=&1_{I} \\
\rho \circ H &=&1_{H}.
\end{eqnarray*}
\end{definition}

\begin{definition}[half-reflection]
A pair of functors%
\begin{equation*}
\xymatrix{\mathbf{A} \ar@<0.5ex>[r]^{I} & \mathbf{B} \ar@<0.5ex>[l]^{G} },
\ \ \ ,\ \ \ \ IG=1_{\mathbf{B}}
\end{equation*}%
is said to be a half-reflection if there is a natural transformation%
\begin{equation*}
\pi :1_{\mathbf{A}}\longrightarrow GI
\end{equation*}%
such that%
\begin{equation*}
I\circ \pi =1_{I}.
\end{equation*}
\end{definition}

\begin{theorem}
For a half-reflection $\left( I,G,\pi \right) $ we always have%
\begin{equation}
\xymatrix{
1
\ar[r]^{\pi}
\ar[rd]_{\pi}
& GI
\ar[d]^{\pi \circ GI}
\\
& GI
}%
.  \label{eq1}
\end{equation}
\end{theorem}

\begin{proof}
By naturality of $\pi $ we have%
\begin{equation*}
\quadrado
{ A  }       { \pi_A  }     {  GIA   }
{ \pi_A  }                  {  GI(\pi_A)   }
{ GIA  }       {  \pi_{GIA} }     {  GI(GIA)   }%
\end{equation*}%
but $GI(GIA)=GIA$ and $GI\left( \pi _{A}\right) =1_{GIA}$.
\end{proof}

When it exists, the natural transformation $\pi :1\longrightarrow GI$ is 
\emph{essentially} unique, in the sense that any other such, say $\pi
^{\prime }:1\longrightarrow GI$ (with $I\circ \pi ^{\prime }=1_{I}$), is of
the form%
\begin{equation*}
\pi _{A}^{\prime }=\pi _{GIA}^{\prime }\pi _{A}.
\end{equation*}
Under which conditions is it really unique?

The name half-reflection is motivated because if instead of $\left( \ref{eq1}%
\right) $ we have $\pi \circ G=1_{G}$ then the result is a reflection.

For any category $\mathbf{B}$ we consider the category of internal reflexive
graphs in $\mathbf{B}$, denoted $RG\left( \mathbf{B}\right) $ as usual:%
\newline
Objects are diagrams in $\mathbf{B}$ of the form%
\begin{equation*}
\xymatrix{
C_1
\ar@<1ex>[r]^{d}
\ar@<-1ex>[r]_{c}
& C_0
\ar[l]|{e}
}%
\ \ ,\ \ de=1=ce;
\end{equation*}%
Morphisms are pairs $\left( f_{1},f_{0}\right) $ making the obvious squares
commutative in the following diagram%
\begin{equation*}
\xymatrix{
 C_1
\ar@<1ex>[r]^{d}
\ar@<-1ex>[r]_{c}
\ar[d]_{f_1}
& C_0
\ar[d]_{f_0}
\ar[l]|{e}
\\
 C'_1
\ar@<1ex>[r]^{d}
\ar@<-1ex>[r]_{c}
& C'_0
\ar[l]|{e'}
}%
.
\end{equation*}%
We will also consider the category of internal precategories in $\mathbf{B}$%
, denoted $PC\left( \mathbf{B}\right) $, where objects are diagrams of the
form%
\begin{equation*}
\xymatrix{
C_2
\ar@<2ex>[r]^{\pi_2}
\ar@<-2ex>[r]_{\pi_1}
\ar[r]|{m}
& C_1
\ar@<-1ex>[l]|{e_2}
\ar@<1ex>[l]|{e_1}
\ar@<1ex>[r]^{d}
\ar@<-1ex>[r]_{c}
& C_0
\ar[l]|{e}
}%
\end{equation*}%
such that%
\begin{equation}
\xymatrix{
C_2
\ar[r]^{\pi_2}
\ar[d]_{\pi_1}
& C_1
\ar@<1ex>[l]^{e_2}
\ar[d]_{c}
\\
C_1
\ar[r]^{d}
\ar@<-1ex>[u]_{e_1}
& C_0
\ar@<1ex>[l]^{e}
\ar@<-1ex>[u]_{e}
}
\label{split quare 1}
\end{equation}%
is a split square (i.e. a split epi in the category of split epis), so that
in particular we have%
\begin{equation}
de=1_{C_{0}}=ce  \label{RGraph condition in precat}
\end{equation}%
and furthermore, the following three conditions are satisfied%
\begin{gather}
dm=d\pi _{2}  \label{dom_preserved in a precat} \\
cm=c\pi _{1}.  \label{cod_preserved in a precat}
\end{gather}%
\begin{equation}
me_{1}=1_{C_{1}}=me_{2};  \label{me1 is 1}
\end{equation}%
and obvious morphisms.

A precategory in this sense becomes a category\footnote{%
In fact it is not quite a category because we are not considering
associativity; also the term precategory is often used when $\left(\ref{me1
is 1}\right) $ is not present; we also observe that in many interesting
cases (for example in Mal'cev categories) assuming only $\left(\ref{me1 is 1}%
\right) $ and the fact that $\left(\ref{split quare 1}\right) $ is a
pullback, then the resulting structure is already an internal category.}  if
the top and left square in $\left( \ref{split quare 1}\right) $ is a
pullback.

\begin{definition}
A category is said to have coequalizers of reflexive graphs if for every
reflexive graph%
\begin{equation*}
\xymatrix{
C_1
\ar@<1ex>[r]^{d}
\ar@<-1ex>[r]_{c}
& C_0
\ar[l]|{e}
}%
\ \ ,\ \ de=1=ce
\end{equation*}%
the coequalizer of $d$ and $c$ exists.
\end{definition}

\begin{definition}[pointed category]
A pointed category is a category enriched in pointed sets. More
specifically, for every pair $X,Y$ of objects, there is a specified
morphism, $0_{X,Y}:X\longrightarrow Y$ with the following property:%
\begin{equation*}
X\overset{0_{X,Y}}{\longrightarrow }Y\overset{f}{\longrightarrow }Z\overset{%
0_{Z,W}}{\longrightarrow }W
\end{equation*}%
\begin{eqnarray}
0_{Z,W}f &=&0_{Y,W}  \label{eq2} \\
f0_{X,Y} &=&0_{X,Z}.  \label{eq3}
\end{eqnarray}
\end{definition}

\begin{definition}[additive category]
An additive category is an $Ab$-category with binary biproducts.
\end{definition}

Observe that on the contrary to the usual practice we are not considering
the existence of a null object, neither in pointed nor in additive
categories.

\section{Additivity}

Let $\mathbf{B}$ be any category and consider the pair of functors%
\begin{equation*}
\xymatrix{\mathbf{B\times B} \ar@<0.5ex>[r]^{I} & \mathbf{B} \ar@<0.5ex>[l]^{G} }
\end{equation*}%
with $I\left( X,B\right) =B$ and $G\left( B\right) =\left( B,B\right) .$

\begin{theorem}
The above pair $\left( I,G\right) $ is a half-reflection if and only if the
category $\mathbf{B}$ is pointed.
\end{theorem}

\begin{proof}
If $\mathbf{B}$ is pointed simply define%
\begin{equation*}
\pi _{\left( X,B\right) }:\left( X,B\right) \longrightarrow \left( B,B\right)
\end{equation*}%
as $\pi _{\left( X,B\right) }=\left( 0_{X,B},1_{B}\right) .$\newline
Now suppose there is a natural transformation%
\begin{equation*}
\pi :1_{\mathbf{BxB}}\longrightarrow GI,
\end{equation*}%
such that $I\circ \pi =1_{I}$, this is the same as having for every pair $%
X,B $ in $\mathbf{B}$ a specified morphism%
\begin{equation*}
\pi _{X,B}:X\longrightarrow B
\end{equation*}%
and conditions $\left( \ref{eq2}\right) $ and $\left( \ref{eq3}\right) $
follow by naturality:%
\begin{equation*}
\quadrado
{ (Y,W)  }       { (\pi_{Y,W} , 1)  }     {  (W,W)   }
{ (f,1)  }                  {  (1,1)   }
{ (Z,W)  }       { (\pi_{Z,W}, 1)  }     {  (W,W)   }%
,\ \ \ \ 
\quadrado
{ (X,Y)  }       { (\pi_{X,Y} , 1)  }     {  (Y,Y)   }
{ (1,f)  }                  {  (f,f)   }
{ (X,Z)  }       { (\pi_{X,Z}, 1)  }     { ( Z,Z)   }%
.
\end{equation*}
\end{proof}

\begin{theorem}
The functor $G$ as above admits a left adjoint if and only if the category $%
\mathbf{B}$ has binary coproducts.
\end{theorem}

\begin{proof}
As it is well known, coproducts are obtained as the left adjoint to the
diagonal functor.
\end{proof}

\begin{theorem}
Given a half-reflection $\left( I,G,\pi \right) $%
\begin{equation*}
\xymatrix{\mathbf{A} \ar@<0.5ex>[r]^{I} & \mathbf{B} \ar@<0.5ex>[l]^{G} } \ \ \ ,\ \ \ \ \pi :1_{\mathbf{A}}\longrightarrow GI,
\end{equation*}%
if the functor $G$ admits a left adjoint%
\begin{equation*}
\left( F,G,\eta ,\varepsilon \right) ,
\end{equation*}%
then, there is a canonical functor%
\begin{equation*}
\mathbf{A}\longrightarrow Pt\left( \mathbf{B}\right)
\end{equation*}%
sending an object $A\in \mathbf{A}$ to the split epi%
\begin{equation*}
\xymatrix{
FA
\ar@<0.5ex>[r]^{\varepsilon_{IA} F(\pi_A)}
& IA
\ar@<0.5ex>[l]^{I(\eta_A)}
}%
.
\end{equation*}
\end{theorem}

\begin{proof}
We only have to prove%
\begin{equation*}
\varepsilon _{IA}F\left( \pi _{A}\right) I\left( \eta _{A}\right) =1_{IA}.
\end{equation*}%
Start with%
\begin{equation*}
\pi _{A}=G\left( \varepsilon _{IA}F\left( \pi _{A}\right) \right) \eta _{A}
\end{equation*}%
and apply $I$ to both sides to obtain%
\begin{equation*}
I\left( \pi _{A}\right) =\varepsilon _{IA}F\left( \pi _{A}\right) I\left(
\eta _{A}\right) ,
\end{equation*}%
by definition we have $I\left( \pi _{A}\right) =1_{IA}$.
\end{proof}

In particular if $\mathbf{B}$ is pointed and has binary coproducts we have
the canonical functor%
\begin{equation*}
\mathbf{B\times B}\overset{T}{\longrightarrow }Pt\left( \mathbf{B}\right)
\end{equation*}%
sending a pair $\left( X,B\right) $ to the split epi%
\begin{equation*}
\xymatrix{
X \sqcup B
\ar@<0.5ex>[r]^{[0 \, 1]}
& B
\ar@<0.5ex>[l]^{\iota_2}
}%
.
\end{equation*}

\begin{theorem}
Let $\mathbf{B}$ be a pointed category with binary coproducts. The canonical
functor $\mathbf{B\times B}\overset{T}{\longrightarrow }Pt\left( \mathbf{B}%
\right) $ admits a right adjoint, $S$, such that $IS=I^{\prime }$%
\begin{equation*}
\xymatrix{
\mathbf{B\times B}   
\ar[rd]_{I}
& & Pt(\mathbf{B})
\ar[ll]_{S}
\ar[ld]^{I'}
\\
& \mathbf{B}
}%
\end{equation*}%
if and only if the category $\mathbf{B}$ has kernels of split epis.\newline
The functor $I^{\prime }$ sends a split epi $\left( A,\alpha ,\beta
,B\right) $ to $B$.
\end{theorem}

\begin{proof}
If the category has kernels of split epis, then for every split epi we
choose a specified kernel%
\begin{equation*}
\xymatrix{
X \ar[r]^{k}  & A \ar@<0.5ex>[r]^{\alpha} & B \ar@<0.5ex>[l]^{\beta}
}%
\end{equation*}%
and the functor $S$, sending $\left( A,\alpha ,\beta ,B\right) $ to the pair 
$\left( X,B\right) $ is the right adjoint for $T$:%
\begin{equation*}
\xymatrix{
(Y,D)
\ar[d]^{(f,g)}
&
Y
\ar[r]^{\iota_1}
\ar[d]_{f}
& Y \sqcup D
\ar@<0.5ex>[r]^{[0 \, 1]}
\ar[d]^{[kf \, \beta g]}
& B
\ar@<0.5ex>[l]^{\iota_2}
\ar[d]^{g}
\\
(X,B)
&
X \ar[r]^{k}  & A \ar@<0.5ex>[r]^{\alpha} & B \ar@<0.5ex>[l]^{\beta}
}%
.
\end{equation*}%
\newline
Now suppose $S$ is a right adjoint to $T$ and it is such that a split epi $%
\left( A,\alpha ,\beta ,B\right) $ goes to a pair of the form%
\begin{equation*}
\left( K[\alpha ],B\right)
\end{equation*}%
with unit and counit as follows%
\begin{equation*}
\xymatrix{
(X,B)
\ar[d]^{(\eta_X,1)}
&
K[\alpha]
\ar[r]^{\iota_1}
\ar[rd]_{\varepsilon_1}
& K[\alpha] \sqcup B
\ar@<0.5ex>[r]^{[0 \, 1]}
\ar[d]^{[\varepsilon_1 \, \beta]}
& B
\ar@<0.5ex>[l]^{\iota_2}
\ar[d]^{1}
\\
(K[0 \,1],B)
&
& A \ar@<0.5ex>[r]^{\alpha} & B \ar@<0.5ex>[l]^{\beta}
}%
.
\end{equation*}%
We have to prove that $\varepsilon _{1}=\ker \alpha $, and in fact, we have $%
\alpha \varepsilon _{1}=0$ and by the universal property of the counit we
have that given a morphism of split epis%
\begin{equation*}
\xymatrix{
X
\ar[r]^{\iota_1}
\ar[rd]_{f}
& X \sqcup B
\ar@<0.5ex>[r]^{[0 \, 1]}
\ar[d]^{[f \, \beta]}
& B
\ar@<0.5ex>[l]^{\iota_2}
\ar[d]^{1}
\\
& A \ar@<0.5ex>[r]^{\alpha} & B \ar@<0.5ex>[l]^{\beta}
}%
\end{equation*}%
that is a morphism $f:X\longrightarrow A$ such that $\alpha f=0$, there
exists a unique $\left( f^{\prime },1\right) :\left( X,B\right)
\longrightarrow \left( K[\alpha ],B\right) $ such that 
\begin{equation*}
\lbrack \varepsilon _{1}\ \beta ]\left( f^{\prime }\sqcup 1\right) =[f\
\beta ]
\end{equation*}%
which is equivalent to say $\varepsilon _{1}f^{\prime }=f$. Hence $%
\varepsilon _{1}$ is a kernel for $\alpha $.
\end{proof}

\begin{theorem}
Let $\mathbf{B}$ be a pointed category with binary coproducts and kernels of
split epis. If the canonical adjunction%
\begin{equation*}
\mathbf{B\times B}\overset{T}{\underset{S}{\underleftarrow{\overrightarrow{\
\ \ \ \ \perp \ \ \ \ \ \ }}}}Pt\left( \mathbf{B}\right)
\end{equation*}%
is an equivalence, then:%
\begin{equation*}
RG\left( \mathbf{B}\right) \sim Mor\left( \mathbf{B}\right)
\end{equation*}%
and%
\begin{equation*}
PC\left( \mathbf{B}\right) \sim 2\text{-}Chains\left( \mathbf{B}\right) .
\end{equation*}
\end{theorem}

\begin{proof}
By the equivalence we have that a split epi%
\begin{equation*}
\xymatrix{
C_1
\ar@<0.5ex>[r]^{d}
& C_0
\ar@<0.5ex>[l]^{e}
}%
\ \ \ \ \ \ ,\ \ \ de=1
\end{equation*}%
is of the form%
\begin{equation*}
\xymatrix{
X \sqcup B
\ar@<0.5ex>[r]^{[0  \, 1]}
& B
\ar@<0.5ex>[l]^{\iota_2}
}%
,
\end{equation*}%
and to give a morphism $c:X\sqcup B\longrightarrow B$ such that $c\iota
_{2}=1$ is to give a morphism%
\begin{equation*}
h:X\longrightarrow B.
\end{equation*}%
So that a reflexive graph is, up to isomorphism, of the form%
\begin{equation*}
\xymatrix{
X \sqcup B
\ar@<1ex>[r]^{[0 \, 1]}
\ar@<-1ex>[r]_{[h \, 1]}
& B
\ar[l]|{\iota_2}
}%
.
\end{equation*}%
To investigate a precategory we observe that the square $\left( \ref{split
quare 1}\right) $ may be considered as a split epi in the category of split
epis, and hence, it is given up to an isomorphism in the form%
\begin{equation}
\xymatrix{
Y \sqcup (X \sqcup B)
\ar[r]^{[0 \, 1]}
\ar[d]_{a \sqcup [h,1]}
& X \sqcup B
\ar@<1ex>[l]^{\iota_2}
\ar[d]_{[h \, 1]}
\\
X \sqcup B
\ar[r]^{[0 \, 1]}
\ar@<-1ex>[u]_{b \sqcup \iota_2}
& B
\ar@<1ex>[l]^{\iota_2}
\ar@<-1ex>[u]_{\iota_2}
}%
\ \ \ ,\ \ \ ab=1.  \label{split square 2}
\end{equation}%
It follows that $m$, satisfying $m\iota _{2}=1$ is of the form%
\begin{equation*}
Y\sqcup \left( X\sqcup B\right) \overset{[v\ 1]}{\longrightarrow }\left(
X\sqcup B\right)
\end{equation*}%
and hence to give $m$ is to give $v:Y\longrightarrow X\sqcup B$.\newline
Since we also have $\left( \ref{dom_preserved in a precat}\right) $ then $%
[0\ 1]v=0$, and $v$ factors through the kernel of $[0\ 1]$ which is (see
Theorem \ref{Th 2})%
\begin{equation*}
\xymatrix{
X \ar[r]^{\iota_1}  & X \sqcup B \ar@<0.5ex>[r]^{[0 \, 1]} & B \ar@<0.5ex>[l]^{\iota_2}
}%
.
\end{equation*}%
This shows that to give $m$ is to give a morphism%
\begin{equation*}
u:Y\longrightarrow X
\end{equation*}%
and hence $m$ is given as%
\begin{equation*}
m=[\iota _{1}u\ 1]:Y\sqcup \left( X\sqcup B\right) \longrightarrow \left(
X\sqcup B\right) .
\end{equation*}%
Finally we have that condition $\left( \ref{cod_preserved in a precat}%
\right) $ is equivalent to $ha=hu$ and $m\left( b\sqcup \iota _{2}\right) $
is equivalent to $ub=1.$\newline
Conclusion 1: A precategory in $\mathbf{B}$ is completely determined by a
diagram%
\begin{equation*}
\xymatrix{
Y
\ar@<1ex>[r]^{a}
\ar@<-1ex>[r]_{u}
& X
\ar[l]|{b}
\ar[r]^{h}
& B
}%
\end{equation*}%
such that 
\begin{equation*}
ab=1=ub\ ,\ ha=hu.
\end{equation*}%
\newline
Continuing with a further analysis we observe that the resulting diagram is
in particular a reflexive graph and hence it is, up to isomorphism, of the
form%
\begin{equation*}
\xymatrix{
Z \sqcup X
\ar@<1ex>[r]^{[0 \,1]}
\ar@<-1ex>[r]_{[t \, 1]}
& X
\ar[l]|(.4){\iota_2}
\ar[r]^{h}
& B
}%
\end{equation*}%
where $h[0\ 1]=h[t\ 1]$ is equivalent to $ht=0$.\newline
Conclusion 2: A precategory in $\mathbf{B}$ is completely determined by a
2-chain complex%
\begin{equation*}
Y\overset{t}{\longrightarrow }X\overset{h}{\longrightarrow }B\ \ ,\ \ ht=0.
\end{equation*}
\end{proof}

\begin{remark}
In the future we will not assume the canonical functor $T$ to be an
equivalence, and hence the second conclusion will no longer be possible.
However, we will be interested in the study of precategories such that $%
\left( \ref{split quare 1}\right) $ is of the form $\left( \ref{split square
2}\right) $ and in that case, provided that $\iota _{1}$ is the kernel of $%
[0\ 1]$ we still can deduce conclusion 1. Such an example is the category of
pointed sets: see Introduction.
\end{remark}

There is a canonical inclusion of reflexive graphs into precategories, by
sending $h:X\longrightarrow B$ to 
\begin{equation*}
\xymatrix{
X
\ar@<1ex>[r]^{1}
\ar@<-1ex>[r]_{1}
& X
\ar[l]|{1}
\ar[r]^{h}
& B
}%
.
\end{equation*}

\begin{theorem}
If $\mathbf{B}$ has coequalizers of reflexive graphs, then the canonical
functor%
\begin{equation*}
PC\left( \mathbf{B}\right) \underset{V}{\longleftarrow }RG\left( \mathbf{B}%
\right)
\end{equation*}%
has a left adjoint.
\end{theorem}

\begin{proof}
The left adjoint is the following.\newline
Given the precategory%
\begin{equation*}
\xymatrix{
Y
\ar@<1ex>[r]^{a}
\ar@<-1ex>[r]_{u}
& X
\ar[l]|{b}
\ar[r]^{h}
\ar[rd]_{\sigma=coeq}
& B
\\
& & X'
\ar@{-->}[u]_{h'}
}%
\end{equation*}%
construct the coequalizer of $u$ and $a$, say $\sigma $, and consider the
reflexive graph in $\mathbf{B}$ determined by%
\begin{equation*}
h^{\prime }:X^{\prime }\longrightarrow B.
\end{equation*}%
This defines a reflection%
\begin{equation*}
PC\left( \mathbf{B}\right) \overset{U}{\longrightarrow }RG\left( \mathbf{B}%
\right)
\end{equation*}%
with unit%
\begin{equation*}
\xymatrix{
Y
\ar@<1ex>[r]^{a}
\ar@<-1ex>[r]_{u}
\ar[d]_{\sigma u= \sigma a}
& X
\ar[l]|{b}
\ar[r]^{h}
\ar[d]_{\sigma}
& B
\ar@{=}[d]
\\
X'
\ar@<1ex>[r]^{1}
\ar@<-1ex>[r]_{1}
& X'
\ar[l]|{1}
\ar[r]^{h'}
& B
}%
.
\end{equation*}
\end{proof}

Next we characterize a category $\mathbf{B}$, pointed, with binary
coproducts and such that the canonical functor%
\begin{equation*}
\mathbf{B\times B}\longrightarrow Pt\left( \mathbf{B}\right)
\end{equation*}%
is an equivalence.

First observe that:

\begin{theorem}
If $\mathbf{B}$, as above, also has binary products, then it is an additive
category (with kernels of split epis).
\end{theorem}

\begin{proof}
We simply observe that in particular%
\begin{equation*}
\xymatrix{
X \sqcup B
\ar@<0.5ex>[r]^{[0 \, 1]}
\ar[d]_{\cong }
& B
\ar@<0.5ex>[l]^{\iota_2}
\ar@{=}[d]
\\
X \times B
\ar@<0.5ex>[r]^{\pi_2}
& B
\ar@<0.5ex>[l]^{<0,1>}
}%
\ \ \ \ \ \text{and \ \ \ \ \ \ \ }%
\xymatrix{
X \sqcup B
\ar@<0.5ex>[r]^{[0 \, 1]}
\ar[d]_{\cong }
& B
\ar@<0.5ex>[l]^{\iota_2}
\ar@{=}[d]
\\
X \times B
\ar@<0.5ex>[r]^{\pi_2}
& B
\ar@<0.5ex>[l]^{<1,1>}
}%
,
\end{equation*}%
since $X\overset{<1,0>}{\longrightarrow }X\times B$ is a kernel for $\pi
_{2} $. See \cite{GJ2} for more details.
\end{proof}

\begin{theorem}
\label{Th 2}Let $\mathbf{B}$ be pointed with binary coproducts and kernels
of split epis. The following conditions are equivalent:

\begin{description}
\item[(a)] the canonical adjunction%
\begin{equation*}
\mathbf{B\times B}\overset{T}{\underset{S}{\underleftarrow{\overrightarrow{\
\ \ \ \ \perp \ \ \ \ \ \ }}}}Pt\left( \mathbf{B}\right)
\end{equation*}%
\begin{eqnarray*}
T\left( X,B\right) &=&\left( X\sqcup B,[0\ 1],\iota _{2},B\right) \\
S\left( A,\alpha ,\beta ,B\right) &=&\left( K[\alpha ],B\right) ,
\end{eqnarray*}%
is an equivalence of categories;

\item[(b)] the category $\mathbf{B}$ satisfies the following two axioms:

\begin{description}
\item[(A1)] for every diagram of the form%
\begin{equation}
\xymatrix{
X \ar[r]^{\iota_1}  & X \sqcup B \ar@<0.5ex>[r]^{[0 \, 1]} & B \ar@<0.5ex>[l]^{\iota_2}
}
\label{A1}
\end{equation}%
the morphism $\iota _{1}$ is the kernel of $[0\ 1]$;

\item[(A2)] the split short five lemma holds, that is, given any diagram of
split epis and respective kernels%
\begin{equation}
\xymatrix{
X
\ar[r]^{k}
\ar[d]_{f}
 & A \ar@<0.5ex>[r]^{\alpha}
\ar[d]_{h}
& B \ar@<0.5ex>[l]^{\beta}
\ar[d]^{g}
\\
X' \ar[r]^{k'}  & A' \ar@<0.5ex>[r]^{\alpha'} & B' \ar@<0.5ex>[l]^{\beta'}
}
\label{A2}
\end{equation}%
if $g$ and $f$ are isomorphisms then $h$ is an isomorphism.
\end{description}
\end{description}
\end{theorem}

\begin{proof}
$(b)\Rightarrow (a)$ Using only (A1) we have that $ST\cong 1$, and using
(A1) and (A2) we have, in particular, that $[k\ \beta ]$ as in%
\begin{equation*}
\xymatrix{
X
\ar[r]^{\iota_1}
\ar@{=}[d]
 & X \sqcup B
\ar@<0.5ex>[r]^{[0 \, 1]}
\ar[d]_{[k \, \beta]}
& B
\ar@<0.5ex>[l]^{\iota_2}
\ar@{=}[d]
\\
X \ar[r]^{k}  & A \ar@<0.5ex>[r]^{\alpha} & B \ar@<0.5ex>[l]^{\beta}
}%
\end{equation*}%
is an isomorphism, and hence $TS\cong 1$.

$(a)\Rightarrow (b)$ Suppose $ST\cong 1$, this gives (A1); suppose $TS\cong
1 $, so that from $\left( \ref{A2}\right) $ we can form%
\begin{equation*}
\xymatrix{
X
\ar[r]^{\iota_1}
\ar[d]_{f}
 & X \sqcup B
\ar@<0.5ex>[r]^{[0 \, 1]}
\ar[d]_{f \sqcup g}
& B
\ar@<0.5ex>[l]^{\iota_2}
\ar[d]^{g}
\\
X' \ar[r]^{\iota_1}  & X \sqcup B \ar@<0.5ex>[r]^{[0 \, 1]} & B \ar@<0.5ex>[l]^{\iota_2}
}%
\end{equation*}%
and if $f$, $g$ are isomorphisms, we can find $h^{-1}=[k\ \beta ]\left(
f^{-1}\sqcup g^{-1}\right) [k^{\prime }\ \beta ^{\prime }]^{-1}$.
\end{proof}

\begin{corollary}
If $T$ is a reflection then it is an equivalence of categories.
\end{corollary}

We may now state the following results.

\begin{conclusion}
Let $\mathbf{B}$ be a pointed category with binary coproducts. TFAE:\newline
(a) T is a reflection and $\mathbf{B}$ has binary products;\newline
(b) $\mathbf{B}$ is additive and has kernels of split epis.
\end{conclusion}

\begin{conclusion}
Let $\mathbf{B}$ be pointed, with binary products and coproducts and kernels
of split epis. TFAE:\newline
(a) $T$ is a reflection;\newline
(b) $\mathbf{B}$ is additive.
\end{conclusion}

\subsection{Restriction to split epis}

Suppose now that the canonical functor $T$ is not an equivalence, but we
still have axiom $\left( \ref{A1}\right) $, that is $ST\cong 1$. The results
relating precategories and reflexive graphs will still hold if we restrict $%
PC\left( \mathbf{B}\right) $ to diagrams of the form%
\begin{equation*}
\xymatrix@C=3pc{
Y \sqcup (X \sqcup B)
\ar@<2ex>[r]^{[0 \, 1]}
\ar@<-2ex>[r]_{a \sqcup c}
\ar[r]|{m}
& X \sqcup B
\ar@<-1ex>[l]|{\iota_2}
\ar@<1ex>[l]|{b \sqcup c}
\ar@<1ex>[r]^{[0 \, 1]}
\ar@<-1ex>[r]_{c}
& B
\ar[l]|{\iota_2}
}%
.
\end{equation*}%
This result will be proved in a more general case in the next sections.

An example of such a case is the category of pointed sets.

If starting with a general half-reflection%
\begin{equation*}
\mathbf{A}\overset{I}{\underset{G}{\underleftarrow{\overrightarrow{\ \ \ \ \
\ \ \ \ \ \ }}}}\mathbf{B\ \ \ ,\ \ \ }\pi :1\longrightarrow GI
\end{equation*}%
such that $G$ admits a left adjoint%
\begin{equation*}
\left( F,G,\eta ,\varepsilon \right)
\end{equation*}%
we consider the canonical functor%
\begin{equation*}
\mathbf{A}\overset{T}{\longrightarrow }Pt\left( \mathbf{B}\right)
\end{equation*}%
and ask if it is an equivalence; if not we then ask if it satisfies at least
one of the axioms $\left( \ref{A1}\right) $ or $\left( \ref{A2}\right) $.
For example for $\mathbf{A=B\times B}$ and assuming the constructions as
above, in the case of pointed sets we have $\left( \ref{A1}\right) $ but not 
$\left( \ref{A2}\right) $, while in groups we have $\left( \ref{A2}\right) $
but not $\left( \ref{A1}\right) $.

In the case where we have only $\left( \ref{A1}\right) $ we will be
interested in the study of $RG\left( \mathbf{B}\right) $ and $PC\left( 
\mathbf{B}\right) $ restricted to split epis of the form%
\begin{equation*}
\xymatrix{
FA
\ar@<0.5ex>[r]^{\varepsilon_{IA} F(\pi_A)}
& IA
\ar@<0.5ex>[l]^{I(\eta_A)}
}%
,
\end{equation*}%
while if in the presence of $\left( \ref{A2}\right) $, but not $\left( \ref%
{A1}\right) $, we may construct a category of internal actions as suggested
in \cite{GJ1}.

\section{Semi-Additivity}

Let $\mathbf{B}$ be a pointed category with binary coproducts and kernels of
split epis. As shown in the previous section there is a canonical adjunction%
\begin{equation}
\mathbf{B\times B}\overset{T}{\underset{S}{\underleftarrow{\overrightarrow{\
\ \ \ \ \ \perp \ \ \ \ \ }}}}Pt\left( \mathbf{B}\right) .
\label{canonical adjunction BxB--->Pt(B)}
\end{equation}%
We are considering $\mathbf{B\times B}$ and $Pt\left( \mathbf{B}\right) $ as
objects in the category of functors over $\mathbf{B}$, that is%
\begin{equation*}
\xymatrix{
\mathbf{B\times B}
\ar[rd]_{I}
& & Pt(\mathbf{B})
\ar[ld]^{I'}
\\
& \mathbf{B}
}%
\end{equation*}%
where $I\left( X,B\right) =B$ and $I^{\prime }\left( A,\alpha ,\beta
,B\right) =B$.

We are also interested in the fact that $I$ is a half-reflection, with
respect to some functor $G$. In the case of $\mathbf{B\times B}$ there is a
canonical choice for $G$, namely the diagonal functor, and it is a
half-reflection if and only if $\mathbf{B}$ is pointed. We are also
interested in the fact that $G$ admits a left adjoint.

In the case of $Pt\left( \mathbf{B}\right) $ there are apparently many good
choices for the functor $G^{\prime }$ to be a half-reflection together with $%
I^{\prime }$. Nevertheless, if we ask that the left adjoint for $G^{\prime }$
to be $F^{\prime }$, such that $F^{\prime }\left( A,\alpha ,\beta ,B\right)
=A$, then we calculate $G^{\prime }$ as follows.

\begin{theorem}
\label{Theorem 1}Let $\mathbf{B}$ be any category and consider the two
functors%
\begin{equation*}
Pt\left( \mathbf{B}\right) \overset{I}{\underset{F}{\underrightarrow{%
\overrightarrow{\ \ \ \ \ \ \ \ \ \ \ }}}}\mathbf{B}
\end{equation*}%
\begin{eqnarray*}
I\left( A,\alpha ,\beta ,B\right) &=&B \\
F\left( A,\alpha ,\beta ,B\right) &=&A.
\end{eqnarray*}%
The functor $F$ admits a right adjoint%
\begin{equation*}
\left( F,G,\eta ,\varepsilon \right)
\end{equation*}%
such that $IG=1_{\mathbf{B}}$ if and only if the category $\mathbf{B}$ has
an endofunctor%
\begin{equation*}
G_{1}:\mathbf{B}\longrightarrow \mathbf{B}
\end{equation*}%
and natural transformations%
\begin{equation*}
\xymatrix{
G_1(B)
\ar@<1ex>[r]^{\pi_B}
\ar@<-1ex>[r]_{\varepsilon_B}
& B
\ar[l]|(.4){\delta_B}
}%
\ \ \ \ ,\ \ \ \ \ \pi _{B}\delta _{B}=1_{B},
\end{equation*}%
satisfying the following property:\newline
for every diagram in $\mathbf{B}$ of the form%
\begin{equation*}
\xymatrix{
A
\ar@<0.5ex>[r]^{\alpha}
\ar[rd]_{f}
& B'
\ar@<0.5ex>[l]^{\beta}
\\
& B
}%
\ \ \ ,\ \ \ \alpha \beta =1
\end{equation*}%
there exists a unique morphism%
\begin{equation*}
f^{\prime }:A\longrightarrow G_{1}\left( B\right)
\end{equation*}%
such that%
\begin{eqnarray*}
\varepsilon _{B}f^{\prime } &=&f \\
\delta _{B}\pi _{B}f^{\prime } &=&f^{\prime }\beta \alpha.
\end{eqnarray*}
\end{theorem}

\begin{proof}
Suppose we have $G_{1},\pi ,\delta ,\varepsilon $ satisfying the required
conditions in the Theorem, then the functor%
\begin{equation*}
G(B)=%
\xymatrix{
G_1(B)
\ar@<0.5ex>[r]^{\pi_B}
&  B
\ar@<0.5ex>[l]^{\delta_B}
}%
\end{equation*}%
is a right adjoint to $F$; in fact (see \cite{ML}, p.83, Theorem 2, (iii))
we have functors $F$ and $G$, and a natural transformation $\varepsilon : FG
\longrightarrow 1$, such that each $\varepsilon_B : FG(B) \longrightarrow B$
is universal from $F$ to $B$:%
\begin{equation*}
\xymatrix{
A
\ar[d]^{f}
\ar@{}[rd]|{:}
&
A
\ar@<0.5ex>[r]^{\alpha}
\ar[d]^{f_1}
\ar@{-->}[rd]|{f}
& B
\ar@<0.5ex>[l]^{\beta}
\ar[d]^{f_0}
\\
B'
&
G_1(B') \ar@<0.5ex>[r]^{\pi_B} & B' \ar@<0.5ex>[l]^{\delta_B}
}%
\end{equation*}%
given $f$, there is a unique $f_{1}$ (with $\varepsilon _{B}f^{\prime } = f
, \, \delta _{B}\pi _{B}f^{\prime } = f^{\prime }\beta \alpha$) and $f_{0}$
follows as $f_{0}=\pi _{B}f_{1}\beta $; conversely, given $f_{1}$, we find $%
f=\varepsilon _{B}f_{1}$.

Now, given an adjunction%
\begin{equation*}
\left( F,G,\eta ,\varepsilon \right)
\end{equation*}%
such that $IG=1$, if writing%
\begin{equation*}
G(B)=%
\xymatrix{
G_1(B)
\ar@<0.5ex>[r]^{G_2(B)}
&  B
\ar@<0.5ex>[l]^{G_3(B)}
}%
\end{equation*}%
we define%
\begin{equation*}
G_{1}=FG\ \ ,\ \ \pi _{B}=G_{2}\left( B\right) \ \ ,\ \ \delta
_{B}=G_{3}\left( B\right)
\end{equation*}%
and%
\begin{equation*}
\varepsilon _{B}:FG\left( B\right) \longrightarrow B
\end{equation*}%
is the counit of the adjunction.

Clearly we have natural transformations with $\pi _{B}\delta _{B}=1$.

It remains to check the stated property - but it is simply the universal
property of $\varepsilon_B$: given a diagram%
\begin{equation*}
\xymatrix{
A
\ar@<0.5ex>[r]^{\alpha}
\ar[rd]_{f}
& B'
\ar@<0.5ex>[l]^{\beta}
\\
& B
}%
\ \ \ ,\ \ \ \alpha \beta =1
\end{equation*}%
there is a unique morphism of split epis%
\begin{equation*}
\xymatrix{
A
\ar@<0.5ex>[r]^{\alpha}
\ar[d]^{f_1}
& B
\ar@<0.5ex>[l]^{\beta}
\ar[d]^{f_0}
\\
G_1(B') \ar@<0.5ex>[r]^{\pi_B} & B' \ar@<0.5ex>[l]^{\delta_B}
}%
\end{equation*}%
such that $\varepsilon _{B}f_{1}=f$; being a morphism of split epis means
that $f_0 = \pi_B f_1 \beta$, and $f_1$ is such that $\delta_B \pi_B f_1 =
f_1 \beta \alpha$.
\end{proof}

If $\mathbf{B}$ has binary products, then for every $B\in \mathbf{B}$, 
\begin{equation*}
\xymatrix@=3pc{
B \times B
\ar@<1ex>[r]^{\pi_2}
\ar@<-1ex>[r]_{\pi_1}
& B
\ar[l]|(.4){<1,1>}
}%
\end{equation*}%
satisfies the required conditions and hence $F$ has a right adjoint, $G$,
sending the object $B$ to the split epi%
\begin{equation*}
\xymatrix{
B \times B
\ar@<0.5ex>[r]^{\pi_2}
& B
\ar@<0.5ex>[l]^{<1,1>}
}%
.
\end{equation*}
And furthermore, in this case the pair $(I,G)$ is a half-reflection. For the
general case, if we ask for $(I,G)$ to be a half-reflection, then the
following result suffices.

\begin{corollary}
\label{Corolary1}Let $\mathbf{B}$ be any category and $I,F:Pt\left( \mathbf{B%
}\right) \longrightarrow \mathbf{B}$ as above. If the category $\mathbf{B}$
is equipped with an endofunctor $G_{1}:\mathbf{B}\longrightarrow \mathbf{B}$
and natural transformations%
\begin{equation*}
\xymatrix{
G_1(B)
\ar@<1ex>[r]^{\pi_B}
\ar@<-1ex>[r]_{\varepsilon_B}
& B
\ar[l]|(.4){\delta_B}
}%
\ \ \ \ ,\ \ \ \ \ \pi _{B}\delta _{B}=1_{B}=\varepsilon _{B}\delta _{B},
\end{equation*}%
satisfying the following property:\newline
for every diagram in $\mathbf{B}$ of the form%
\begin{equation*}
\xymatrix{
A
\ar[rd]_{f}
& A
\ar@<0.5ex>[l]^{t}
\\
& B
}%
\ \ \ ,\ \ \ t^{2}=t
\end{equation*}%
there exists a unique morphism%
\begin{equation*}
f^{\prime }:A\longrightarrow G_{1}\left( B\right)
\end{equation*}%
such that%
\begin{equation*}
\pi _{B}f^{\prime }=ft\ ,\ \varepsilon _{B}f^{\prime }=f\ ,\ f^{\prime
}t=\delta _{B}ft,
\end{equation*}
then the functor $F$ has a right adjoint, say $G$, and the pair $(I,G)$ is a
half reflection.
\end{corollary}

\begin{proof}
It is clear that the property is sufficient to obtain $G$ as a right adjoint
to $F$ as in the previous Theorem, simply considering $t = \beta\alpha$ and
observing that the two conditions $\pi_B f^{\prime}= f t$, $%
f^{\prime}t=\delta_B f t$ give $\delta_B\pi_B f^{\prime}= f^{\prime}\beta
\alpha$: start with $\pi_B f^{\prime}= f t$, precompose with $\delta_B$, and
replace $\delta_B f t$ by $f^{\prime}t$.

As a consequence we have that 
\begin{equation*}
\pi_B f^{\prime}\beta = \varepsilon_B f^{\prime}\beta ,
\end{equation*}
since 
\begin{eqnarray*}
ft=ftt=\pi_B f^{\prime}t= \pi_B f^{\prime}\beta \alpha = \pi_B
f^{\prime}\beta \\
ft=\varepsilon_B \delta_B f t = \varepsilon_B f^{\prime}t= \varepsilon_B
f^{\prime}\beta \alpha =\varepsilon_B f^{\prime}\beta
\end{eqnarray*}
and hence, given $f:A\longrightarrow B$, we have $(f_1,f_0)$, with $%
f_1=f^{\prime}$ given by the \emph{universal} property and $f_0=f\beta$ ($%
=\pi_B f^{\prime}\beta = \varepsilon_B f^{\prime}\beta$).

The pair $\left( I,G\right) $ is a half-reflection with 
\begin{equation*}
\pi :1_{Pt\left( \mathbf{B}\right) }\longrightarrow GI
\end{equation*}%
given by%
\begin{equation*}
\xymatrix{
A
\ar@<0.5ex>[r]^{\alpha}
\ar@{-->}[d]^{\delta_B \alpha}
& B
\ar@<0.5ex>[l]^{\beta}
\ar@{=}[d]
\\
G_1(B) \ar@<0.5ex>[r]^{\pi_B} & B \ar@<0.5ex>[l]^{\delta_B}
}%
,
\end{equation*}%
and furthermore this is the only possibility.
\end{proof}

\begin{corollary}
In the conditions of the above Corollary (and assuming $\mathbf{B}$ is
pointed), the kernel of $\pi _{B}$ is the morphism induced by the diagram%
\begin{equation*}
\xymatrix{
B
\ar[rd]_{1}
& B
\ar@<0.5ex>[l]^{0}
\\
& B
}%
.
\end{equation*}
\end{corollary}

As mentioned in the previous section, if the canonical adjunction $\left( %
\ref{canonical adjunction BxB--->Pt(B)}\right) $ is not an equivalence we
are interested in considering bigger categories, $\mathbf{A}$, that we will
call categories of actions, in the place of $\mathbf{B\times B}$, in order
to obtain an equivalence of categories $\mathbf{A}\sim Pt\left( \mathbf{B}%
\right) $.

We now turn our attention to the category of internal actions in $\mathbf{B}$%
.

To define the category of internal actions in $\mathbf{B}$, in the sense of 
\cite{GJ1}, we only need to assume $\mathbf{B}$ to be pointed, with binary
coproducts and kernels of split epis: exactly the same conditions necessary
to consider the canonical adjunction $\left( \ref{canonical adjunction
BxB--->Pt(B)}\right) $; and the construction of the category of internal
actions is actually suggested by the adjunction. This seems to suggest an
iterative process to obtain bigger and bigger categories \textquotedblleft
of actions", $\mathbf{A}_{1},$ $\mathbf{A}_{2}\ ,...$.

\subsection{The category of internal actions}

Let $\mathbf{B}$ be a pointed category with binary coproducts and kernels of
split epis. The category of internal actions in $\mathbf{B}$, denoted $%
Act\left( \mathbf{B}\right) $, is defined as follows.

Objects are triples $\left( X,\xi ,B\right) $ where $X$ and $B$ are objects
in $\mathbf{B}$ and $\xi :B\flat X\longrightarrow X$ is a morphism such that 
\begin{eqnarray*}
\xi \eta _{X} &=&1 \\
\xi \mu _{X} &=&\xi \left( 1\flat \xi \right)
\end{eqnarray*}%
where the object $B\flat X$ is the kernel, $k:B\flat X\longrightarrow
X\sqcup B$ of $[0,1]:X\sqcup B\longrightarrow B$ and $\eta _{X}$, $\mu _{x}$
are induced, respectively, by $\iota _{1}$ and $[k\ \iota _{2}]$. See \cite%
{GJ1} for more details.

We now have to consider $Act\left( \mathbf{B}\right) $ as an
half-reflection, $\left( I,G\right) $ (with $G$ admitting a left adjoint),
over $\mathbf{B}$. Clearly we have a functor%
\begin{equation*}
\xymatrix{
\mathbf{A}
\ar[rd]_{I}
& & Pt(\mathbf{B})
\ar[ld]^{I'}
\ar[ll]_{S}
\\
& \mathbf{B}
}%
\end{equation*}%
sending a split epi $\left( A,\alpha ,\beta ,B\right) $ to $\left( X,\xi
,B\right) $ as suggested in the following diagram%
\begin{equation*}
\xymatrix{
B \flat X
\ar[r]^{k'}
\ar@{-->}[d]_{\xi}
 & X \sqcup B \ar@<0.5ex>[r]^{[0 \, 1]}
\ar[d]_{[k \, \beta]}
& B \ar@<0.5ex>[l]^{\iota_2}
\ar@{=}[d]
\\
X \ar[r]^{k}  & A \ar@<0.5ex>[r]^{\alpha} & B \ar@<0.5ex>[l]^{\beta}
}%
.
\end{equation*}%
It is well defined because $k$ (being a kernel) is monic and%
\begin{equation*}
k\xi \eta _{X}=[k\ \beta ]k^{\prime }\eta _{X}=[k\ \beta ]\iota _{1}=k
\end{equation*}%
so that $\xi \eta _{X}=1$; a similar argument shows $\xi \mu _{X}=\xi \left(
1\flat \xi \right) $.

To obtain a functor $G:\mathbf{B}\longrightarrow Act\left( \mathbf{B}\right) 
$ we compose%
\begin{equation*}
\mathbf{B}\longrightarrow Pt\left( \mathbf{B}\right) \longrightarrow
Act\left( \mathbf{B}\right)
\end{equation*}%
where $\mathbf{B}\longrightarrow Pt\left( \mathbf{B}\right) $ is the
half-reflection of Corollary \ref{Corolary1}; the resulting $G$ sends an
object $B\in \mathbf{B}$ to the internal action $\left( B^{\prime }\flat
B,\xi _{B},B\right) $ as suggested by the following diagram%
\begin{equation*}
\xymatrix{
B' \flat B
\ar[r]^{\ker}
\ar@{-->}[d]_{\xi_B}
 & B' \sqcup B \ar@<0.5ex>[r]^{[0 \, 1]}
\ar[d]_{[\ker \, \delta_B]}
& B \ar@<0.5ex>[l]^{\iota_2}
\ar@{=}[d]
\\
B' \ar[r]^{\ker}  & G_1(B) \ar@<0.5ex>[r]^{\pi_B} & B \ar@<0.5ex>[l]^{\delta_B}
}%
.
\end{equation*}%
In the case of Groups this corresponds to the action by conjugation (see 
\cite{GJ1}). The next step is to require that $G$ admits a left adjoint,
which in the case of Groups is true and it corresponds to the construction
of a semi-direct product from a given action.

For convenience, we will now assume the existence of binary products,
instead of the data $G_{1},\pi ,\delta ,\varepsilon $ of Theorem \ref%
{Theorem 1}.

For the rest of this section, and if not explicitly stated otherwise, we
will assume that $\mathbf{B}$ is a pointed category with binary products and
coproducts and kernels of split epis.

With such assumptions we automatically consider the half-reflection%
\begin{equation*}
Act\left( \mathbf{B}\right) \overset{I}{\underset{G}{\underleftarrow{%
\overrightarrow{\ \ \ \ \ \ \ \ \ \ \ }}}}\mathbf{B\ \ ,\ \ }\pi
:1\longrightarrow GI
\end{equation*}%
where $I\left( X,\xi ,B\right) =B$, $G\left( B\right) =\left( B,\xi
_{B},B\right) $ obtained from%
\begin{equation*}
\xymatrix{ B \flat B \ar[r]^{k} \ar@{-->}[d]_{\xi_B} & B \sqcup B
\ar@<0.5ex>[r]^{[0 \, 1]} \ar[d]|{[<1,0> \, <1,1>]} & B
\ar@<0.5ex>[l]^{\iota_2} \ar@{=}[d] \\ B \ar[r]^{<1,0>} & B \times B
\ar@<0.5ex>[r]^{\pi_2} & B \ar@<0.5ex>[l]^{<1,1>} }
\end{equation*}%
(note that $<0,1>$ is the kernel of $\pi _{2}$), and the natural
transformation $\pi :1\longrightarrow GI$ given by%
\begin{equation*}
\xymatrix{
(X,\xi,B)
\ar[d]_{(0,1)}
\\
(B,\xi_B,B)
}%
\ \ \ \ \ \ \ \ \ \ \ 
\quadrado
{ B \flat X  }       { \xi  }     {  X   }
{ 1 \flat 0  }                  {   0  }
{ B \flat B  }       { \xi_B  }     {   B  }%
;
\end{equation*}%
which is well defined because $\xi _{B}\left( 1\flat 0\right) =0$, since $%
\left\langle 1,0\right\rangle \xi _{B}\left( 1\flat 0\right) =\left\langle
0,0\right\rangle $:%
\begin{eqnarray*}
\left\langle 1,0\right\rangle \xi _{B}\left( 1\flat 0\right) =[\left\langle
1,0\right\rangle \ \left\langle 1,1\right\rangle ]\left( 0\sqcup 1\right)
\ker [0\ 1] = \\
=\left\langle [1\ 1],[0\ 1]\right\rangle \left( 0\sqcup 1\right) \ker [0\
1]=\left\langle 0,[0\ 1]\right\rangle \ker [0\ 1] =\left\langle
0,0\right\rangle .
\end{eqnarray*}

\begin{definition}[semi-direct products]
We will say that $\mathbf{B}$ has semidirect products, if the functor $G$
admits a left adjoint.
\end{definition}

Note that this is a weaker notion of Bourn-Janelidze categorical semidirect
products \cite{JB}, since we are not asking for the induced adjuntion
between $Act(\mathbf{B})$ and $Pt(\mathbf{B})$ to be an equivalence of
categories.

We now state a sufficient condition for $\mathbf{B}$ to have semidirect
products.

\begin{theorem}
The functor $G$, in the half-reflection%
\begin{equation*}
Act\left( \mathbf{B}\right) \overset{I}{\underset{G}{\underleftarrow{%
\overrightarrow{\ \ \ \ \ \ \ \ \ \ \ }}}}\mathbf{B\ \ ,\ \ }\pi
:1\longrightarrow GI,
\end{equation*}%
as above, admits a left adjoint if the category $\mathbf{B}$ has
coequalizers of reflexive graphs.
\end{theorem}

\begin{proof}
Given an object $\left( X,\xi ,B\right) $, consider the reflexive graph%
\begin{equation*}
\xymatrix{
(B \flat X) \sqcup B
\ar@<1ex>[r]^{[k \, \iota_2]}
\ar@<-1ex>[r]_{\xi \sqcup 1}
& X \sqcup B
\ar[l]|(.4){\eta \sqcup 1}
}%
.
\end{equation*}%
The left adjoint, $F$, is given by the coequalizer of $[k\ \iota _{2}]$ and $%
\xi \sqcup 1$:%
\begin{equation*}
\xymatrix{
(B \flat X) \sqcup B
\ar@<1ex>[r]^{[k \, \iota_2]}
\ar@<-1ex>[r]_{\xi \sqcup 1}
& X \sqcup B
\ar[l]|(.4){\eta \sqcup 1}
\ar[r]^{\sigma}
& F(X,\xi,B)
}%
.
\end{equation*}%
See \cite{GJ1} for more details.
\end{proof}

Let us from now on assume that $\mathbf{B}$ is a pointed category with
binary products and coproducts and kernels of split epis, and coequalizers
of reflexive graphs.

The next step is to consider the canonical functor%
\begin{equation*}
Act\left( \mathbf{B}\right) \overset{T}{\longrightarrow }Pt\left( \mathbf{B}%
\right)
\end{equation*}%
sending $\left( X,\xi ,B\right) $ to the split epi $\left( F\left( X,\xi
,B\right) ,\overline{[0\ 1]},\sigma \iota _{2}\right) $ where $\overline{[0\
1]}$ is such that $\overline{[0\ 1]}\sigma =[0\ 1]$, investigate whether it
is an equivalence of categories and study internal precategories and
reflexive graphs in $\mathbf{B}$.

First we show that under the given assumptions, it is always an adjunction.

\begin{theorem}
The functors%
\begin{equation*}
Act\left( \mathbf{B}\right) \overset{T}{\underset{S}{\underleftarrow{%
\overrightarrow{\ \ \ \ \ \ \ \ \ \ \ }}}}Pt\left( \mathbf{B}\right)
\end{equation*}%
as defined above, form an adjoint situation.
\end{theorem}

\begin{proof}
Consider the following diagram%
\begin{equation}
\quadrado
{ B \flat X  }       { \xi  }     {  X   }
{ f_0 \flat g  }                  {   g  }
{ B' \flat X'  }       { \xi_{A}'  }     {   X'  }%
\ \ \ \ \ :\ \ \ \ 
\xymatrix{
X
\ar[r]^{\sigma \iota_1}
\ar[d]_{g}
 & F(X,B) \ar@<0.5ex>[r]^{\overline{[0\ 1]}}
\ar[d]_{f_1}
& B \ar@<0.5ex>[l]^{\sigma \iota_2}
\ar[d]^{f_0}
\\
X' \ar[r]^{k}  & A \ar@<0.5ex>[r]^{\alpha} & B' \ar@<0.5ex>[l]^{\beta}
}
\label{diagram1}
\end{equation}%
where $\left( X,\xi ,B\right) $ is an object in $Act\left( \mathbf{B}\right) 
$ and $\left( X^{\prime },\xi _{A}^{\prime },B^{\prime }\right) $ is $%
S\left( A,\alpha ,\beta ,B^{\prime }\right) $.

Given $\left( f_{1},f_{0}\right) $, since $k$ is the kernel of $\alpha $ and 
\begin{equation*}
\alpha f_{1}\sigma \iota _{1}=f_{0}\overline{[0\ 1]}\sigma \iota
_{1}=f_{0}[0\ 1]\iota _{1}=0
\end{equation*}%
we obtain $g$ as the unique morphism such that $kg=f_{1}\sigma \iota _{1}$.

To prove that the pair $\left( g,f_{0}\right) $ is a morphism in $Act\left( 
\mathbf{B}\right) $, that is, the left hand square in $\left( \ref{diagram1}%
\right) $ commutes, we have to show 
\begin{equation*}
\xi _{A}^{\prime }\left( f_{0}\flat g\right) =g\xi
\end{equation*}%
and we do the following: first observe that $k\xi _{A}^{\prime }\left(
f_{0}\flat g\right) =f_{1}\sigma k^{\prime \prime }$, in fact (see diagram
below, where $k^{\prime }$ and $k^{\prime \prime }$ are kernels)%
\begin{equation*}
\xymatrix{
B \flat  X
\ar[rr]^{k''}
\ar[dd]_{f_0 \flat g}
\ar[rd]^{\xi}
& & X \sqcup B
\ar@<0.5ex>[rr]^{[0\ 1]}
\ar[dd]_(.3){g \sqcup f_0}
\ar[rd]^{\sigma}
&& B \ar@<0.5ex>[ll]^{\sigma \iota_2}
\ar[dd]^{f_0}
\ar@{=}[rd]
\\
& X
\ar[rr]^{\sigma \iota_1}
\ar[dd]_{g}
 && F(X,B) \ar@<0.5ex>[rr]^(.7){\overline{[0\ 1]}}
\ar[dd]_{f_1}
&& B \ar@<0.5ex>[ll]^{\sigma \iota_2}
\ar[dd]^{f_0}
\\
B' \flat X'
\ar[rr]^{k''}  
\ar[rd]_{\xi'_A}
&& X' \sqcup B'
\ar@<0.5ex>[rr]^(.7){[0\ 1]} 
\ar[rd]_{[k \, \beta]}
&& B' 
\ar@<0.5ex>[ll]^(.3){\iota_2}
\ar@{=}[rd]
\\
& X' \ar[rr]^{k}  && A \ar@<0.5ex>[rr]^{\alpha} && B' \ar@<0.5ex>[ll]^{\beta}
}%
\end{equation*}%
\begin{eqnarray*}
k\xi _{A}^{\prime }\left( f_{0}\flat g\right) &=&[k\ \beta ]k^{\prime
}\left( f_{0}\flat g\right) \text{\ ,\ definition of }\xi _{A}^{\prime } \\
&=&[k\ \beta ]\left( g \sqcup f_{0}\right) k^{\prime \prime } \\
&=&[kg\ \beta f_{0}]k^{\prime \prime }
\end{eqnarray*}%
and%
\begin{eqnarray*}
f_{1}\sigma k^{\prime \prime } &=&f_{1}\sigma \lbrack \iota _{1}\ \iota
_{2}]k^{\prime \prime } \\
&=&[f_{1}\sigma \iota _{1}\ f_{1}\sigma \iota _{2}]k^{\prime \prime } \\
&=&[kg\ \beta f_{0}]k^{\prime \prime };
\end{eqnarray*}%
we also have $kg\xi =f_{1}\sigma \iota _{1}\xi $,\ by definition of $g$. The
result follows from the fact that $k$ is monic and%
\begin{equation*}
\sigma k^{\prime \prime }=\sigma \iota _{1}\xi ,
\end{equation*}%
which follows from $\left( \ref{equivariance condition}\right) $ by taking $%
f=\sigma \iota _{1}$ and $g=\sigma \iota _{2}$.

Conversely, given $g$ and $f_{0}$ such that the left hand square in $\left( %
\ref{diagram1}\right) $ commutes, we find $f_{1}=\overline{[kg\ \beta f_{0}]}
$, which is well defined because (see $\ref{equivariance condition}$ below)%
\begin{equation*}
kg\xi =[kg\ \beta f_{0}]k^{\prime \prime },
\end{equation*}%
indeed we have%
\begin{eqnarray*}
kg\xi &=&k\xi _{A}^{\prime }\left( f_{0}\flat g\right) \\
&=&[k\ \beta ]k^{\prime }\left( f_{0}\flat g\right) \\
&=&[k\ \beta ]\left( g+f_{0}\right) k^{\prime \prime } \\
&=&[kg\ \beta f_{0}]k^{\prime \prime }.
\end{eqnarray*}
\end{proof}

In what follows we will need the following.

To give a morphism%
\begin{equation*}
F\left( X,\xi ,B\right) \longrightarrow B^{\prime }
\end{equation*}%
is to give a pair $\left( f,g\right) $ with $f:X\longrightarrow B^{\prime }$
and $g:B\longrightarrow B^{\prime }$ such that%
\begin{equation*}
\lbrack f\ g][k\ \iota _{2}]=[f\ g]\left( \xi \sqcup 1\right)
\end{equation*}%
or equivalently%
\begin{equation}
\lbrack f\ g]k=f\xi .  \label{equivariance condition}
\end{equation}%
See \cite{GJ1} \ for more details.

\begin{theorem}
Let $\mathbf{B}$ be a pointed category with binary products and coproducts,
kernels of split epis and coequalizers of reflexive graphs. If the canonical
functor%
\begin{equation*}
Act\left( \mathbf{B}\right) \overset{T}{\longrightarrow }Pt\left( \mathbf{B}%
\right)
\end{equation*}%
is an equivalence, then:%
\begin{equation*}
RG\left( \mathbf{B}\right) \sim \text{Pre-X-Mod}\left( \mathbf{B}\right)
\end{equation*}%
\begin{equation*}
PC\left( \mathbf{B}\right) \sim \text{2-ChainComp}\left( \mathbf{B}\right) .
\end{equation*}%
The objects in Pre-X-Mod$\left( \mathbf{B}\right) $ are pairs $\left( h,\xi
\right) $ with $h:X\longrightarrow B$ \ a morphism in $\mathbf{B}$ and $\xi
:B\flat X\longrightarrow X$ an action $\left( X,\xi ,B\right) $ in $%
Act\left( \mathbf{B}\right) $ satisfying the following condition%
\begin{equation*}
\lbrack h\ 1]k=h\xi ,
\end{equation*}%
with $k:B\flat X\longrightarrow X+B$ the kernel of $[0\ 1];$\newline
The objects in 2-ChainComp$\left( \mathbf{B}\right) $ are sequences%
\begin{equation*}
Z\overset{t}{\longrightarrow }X\overset{h}{\longrightarrow }B\ \ \ ,\ \ \
ht=0
\end{equation*}%
together with actions%
\begin{eqnarray*}
\xi _{X} &:&B\flat X\longrightarrow X \\
\xi _{Z} &:&X\flat Z\longrightarrow Z \\
\xi _{F(Z,X)} &:&F(X,B)\flat F(Z,X)\longrightarrow F(Z,X)
\end{eqnarray*}%
subject to the following conditions%
\begin{eqnarray*}
\lbrack h\ 1]k_{X} &=&h\xi _{X} \\
\lbrack t\ 1]k_{Z} &=&t\xi _{Z} \\
\lbrack \sigma \iota _{1}\overline{[t\ 1]}\ 1]k_{F\left( Z,X\right) }
&=&\sigma \iota _{1}\overline{[t\ 1]}\xi _{F\left( Z,X\right) } \\
\overline{\lbrack 0\ 1]}\xi _{F\left( Z,X\right) } &=&\xi _{X}\left( 
\overline{[h\ 1]}\flat \overline{[0\ 1]}\right) \\
\sigma \iota _{2}\xi _{X} &=&\xi _{F\left( Z,X\right) }\left( \sigma \iota
_{2}\flat \sigma \iota _{2}\right) .
\end{eqnarray*}
\end{theorem}

\begin{proof}
Using the equivalence $T$, a reflexive graph in $\mathbf{B}$ is of the form%
\begin{equation*}
\xymatrix{
F(X,B)
\ar@<1ex>[r]^{\overline{[0 \, 1]}}
\ar@<-1ex>[r]_{c}
& B
\ar[l]|{\sigma \iota_2}
}%
\ \ \ \ c\sigma \iota _{2}=1.
\end{equation*}%
By definition of $F\left( X,B\right) $ we have%
\begin{equation*}
\xymatrix{
X
\ar[r]^{\sigma \iota_1}
\ar[rd]_{h=c \sigma \iota_1}
& F(X,B)
\ar[d]^{c}
& B
\ar[l]_{\sigma \iota_2}
\ar@{=}[ld]
\\
& B
}%
\end{equation*}%
and the pair $\left( h,1\right) $ induces $c=\overline{[h\ 1]}$ if and only
if%
\begin{equation*}
\lbrack h\ 1]k=h\xi .
\end{equation*}

For a precategory, observing that a split square $\left( \ref{split quare 1}%
\right) $ is in fact a split epi in $Pt\left( \mathbf{B}\right) $ and using
the equivalence $Act\left( \mathbf{B}\right) \sim Pt\left( \mathbf{B}\right) 
$ we have that every such split square is of the form%
\begin{equation*}
\xymatrix{
F(Y, F(X,B))
\ar[r]^{\overline{[0 \, 1]}}
\ar[d]_{F(a, \overline{[h,1]})}
& F(X,B)
\ar@<1ex>[l]^{\sigma \iota_2}
\ar[d]_{\overline{[h \, 1]}}
\\
F(X,B)
\ar[r]^{\overline{[0 \, 1]}}
\ar@<-1ex>[u]_{F(b ,\sigma \iota_2)}
& B
\ar@<1ex>[l]^{\sigma \iota_2}
\ar@<-1ex>[u]_{\sigma \iota_2}
}%
,
\end{equation*}%
and hence giving such a split square is to give internal actions $\left(
X,\xi ,B\right) $ and $\left( Y,\xi ^{\prime },F\left( X,B\right) \right) $
together with morphisms $a,b,h$ such that the following squares commute%
\begin{equation*}
\xymatrix{
F(X,B) \flat Y
\ar[r]^{\xi'}
\ar@<-0.5ex>[d]_{\overline{[h \, 1]} \flat a}
& Y
\ar@<-0.5ex>[d]_{a}
\\
B \flat X
\ar@<-0.5ex>[u]_{\sigma \iota_2 \flat b}
\ar[r]^{\xi}
& X
\ar@<-0.5ex>[u]_{b}
\ar[r]^{h}
& B
}%
\end{equation*}%
and%
\begin{equation*}
\lbrack h\ 1]k=h\xi .
\end{equation*}%
It remains to insert the morphism%
\begin{equation*}
m:F\left( Y,F\left( X,B\right) \right) \longrightarrow F\left( X,B\right)
\end{equation*}%
satisfying the following conditions%
\begin{eqnarray}
m\sigma \iota _{2} &=&1  \label{eq4} \\
mF\left( b,\sigma \iota _{2}\right) &=&1  \label{eq5} \\
\overline{\lbrack h\ 1]}m &=&\overline{[h\ 1]}F\left( a,\overline{[h\ 1]}%
\right)  \label{eq6} \\
\overline{\lbrack 0\ 1]}m &=&\overline{[0\ 1]}\overline{[0\ 1]}.  \label{eq7}
\end{eqnarray}%
From $\left( \ref{eq4}\right) $ we conclude that $m=\overline{[v\ 1]}$ for
some $v:Y\longrightarrow F\left( X,B\right) $ such that%
\begin{equation*}
\lbrack v\ 1]k^{\prime }=v\xi ^{\prime }.
\end{equation*}%
Using $\left( \ref{eq7}\right) $ we conclude that $\overline{[0\ 1]}v=0$ so
that $v$ factors through the kernel of $\overline{[0\ 1]}$, which is $\sigma
\iota _{1}$ because $T$ is an equivalence, and finally we have%
\begin{equation*}
m=\overline{[\sigma \iota _{1}u\ 1]}
\end{equation*}%
for some $u:Y\longrightarrow X$ such that%
\begin{equation*}
\lbrack \sigma \iota _{1}u\ 1]k^{\prime }=\sigma \iota _{1}u\xi ^{\prime }.
\end{equation*}%
Condition $\left( \ref{eq5}\right) $ gives $ub=1$ while condition $\left( %
\ref{eq6}\right) $ gives $ha=hu$:%
\begin{align*}
mF\left( b,\sigma \iota _{2}\right) & =1\Leftrightarrow \overline{[\sigma
\iota _{1}u\ 1]}F\left( b,\sigma \iota _{2}\right) =1\Leftrightarrow 
\overline{[\sigma \iota _{1}u\ 1]}\overline{[\sigma \iota _{1}b\ \sigma
\iota _{2}\sigma \iota _{2}]}=1\Leftrightarrow \\
& \Leftrightarrow \overline{[\sigma \iota _{1}u\ 1]}[\sigma \iota _{1}b\
\sigma \iota _{2}\sigma \iota _{2}]=\sigma \Leftrightarrow \lbrack \sigma
\iota _{1}ub\ \ \sigma \iota _{2}]=[\sigma \iota _{1}\ \sigma \iota
_{2}]\Leftrightarrow \\
& \Leftrightarrow \sigma \iota _{1}ub=\sigma \iota _{1}\Leftrightarrow ub=1;
\end{align*}%
\begin{eqnarray*}
\overline{\lbrack h\ 1]}m &=&\overline{[h\ 1]}F\left( a,\overline{[h\ 1]}%
\right) \Leftrightarrow \overline{[h\ 1]}\overline{[\sigma \iota _{1}u\ 1]}=%
\overline{[h\ 1]}\overline{[\sigma \iota _{1}a\ \ \sigma \iota _{2}\overline{%
[h\ 1]}]}\Leftrightarrow \\
&\Leftrightarrow &\overline{[h\ 1]}[\sigma \iota _{1}u\ 1]=\overline{[h\ 1]}%
[\sigma \iota _{1}a\ \ \sigma \iota _{2}\overline{[h\ 1]}]\Leftrightarrow \\
&\Leftrightarrow &[hu\ \overline{[h\ 1]}]=[ha\ \overline{[h\ 1]}%
]\Leftrightarrow hu=ha.
\end{eqnarray*}

Conclusion 1: A precategory in $\mathbf{B}$ is given by the following data%
\begin{equation}
\xymatrix{
F(X,B) \flat Y
\ar[r]^{\xi'}
\ar@<-0.5ex>[d]_{\overline{[h \, 1]} \flat a}
& Y
\ar@<-0.5ex>[d]_{a}
\ar@/^1pc/[d]^{u}
\\
B \flat X
\ar@<-0.5ex>[u]_{\sigma \iota_2 \flat b}
\ar[r]^{\xi}
& X
\ar@<-0.5ex>[u]_{b}
\ar[r]^{h}
& B
}
\label{diagram2}
\end{equation}%
such that \thinspace $\xi $,$\xi ^{\prime }$ are internal actions, the
obvious squares commute, and the following conditions are satisfied%
\begin{eqnarray}
hu &=&ha  \label{eq8 and folowings} \\
ub &=&1=ab  \notag \\
\lbrack \sigma \iota _{1}u\ 1]k^{\prime } &=&\sigma \iota _{1}u\xi ^{\prime }
\notag \\
\lbrack h\ 1]k &=&h\xi .  \notag
\end{eqnarray}

We now continue to investigate it further and replace the split epi $\left(
Y,a,b,X\right) $ with an action $\left( Z,\xi _{Z},B\right) $. For
convenience we will also rename $\xi _{X}:=\xi $, $k_{X}:=k$, $\xi
_{F(Z,X)}:=\xi ^{\prime }$, $k_{F\left( Z,X\right) }=k^{\prime }$. The
diagram $\left( \ref{diagram2}\right) $ becomes%
\begin{equation*}
\xymatrix{
F(X,B) \flat F(Z,X)
\ar[r]^{\xi_{F(Z,X)}}
\ar@<-0.5ex>[d]_{\overline{[h \, 1]} \flat \overline{[0 \, 1]}}
& F(Z,X)
\ar@<-0.5ex>[d]_{\overline{[0 \, 1]}}
\ar@/^2pc/[d]^{\overline{[t \, 1]}}
\\
B \flat X
\ar@<-0.5ex>[u]_{\sigma \iota_2 \flat \sigma \iota_2}
\ar[r]^{\xi_X}
& X
\ar@<-0.5ex>[u]_{\sigma \iota_2}
\ar[r]^{h}
& B
}%
\end{equation*}%
for some $t:Z\longrightarrow X$ such that $[t\ 1]k_{Z}=t\xi _{Z}$. The
commutativity of the appropriate squares in the diagram above plus the
reinterpretation of conditions $\left( \ref{eq8 and folowings}\right) $
gives the stated result.
\end{proof}

\begin{remark}
In order to consider a reflexive graph $\left( h:X\longrightarrow B,\xi
:B\flat X\longrightarrow X\right) $ as a precategory of the form%
\begin{equation*}
\xymatrix{
F(X,B) \flat X
\ar[r]^{\xi (\overline{[h \, 1]} \flat 1)}
\ar@<-0.5ex>[d]_{\overline{[h \, 1]} \flat 1}
& X
\ar@<-0.5ex>[d]_{1}
\ar@/^1pc/[d]^{1}
\\
B \flat X
\ar@<-0.5ex>[u]_{\sigma \iota_2 \flat 1}
\ar[r]^{\xi}
& X
\ar@<-0.5ex>[u]_{1}
\ar[r]^{h}
& B
}%
\end{equation*}%
we need, in addition to $[h\ 1]k=h\xi $ that%
\begin{equation*}
\lbrack \sigma \iota _{1}\ 1]k^{\prime }=\sigma \iota _{1}\xi \left( 
\overline{[h\ 1]}\flat 1\right)
\end{equation*}%
where $k:B\flat X\longrightarrow X\sqcup B$ and $k^{\prime }:F\left(
X,B\right) \flat X\longrightarrow X\sqcup F\left( X,B\right) $ are the
kernels of $[0\ 1].$\newline
In the case of Groups, this corresponds to the Peiffer identity that
distinguishes a precrossed module from a crossed module.
\end{remark}

We now give a characterization of categories $\mathbf{B}$ such that $%
Act\left( \mathbf{B}\right) \sim Pt\left( \mathbf{B}\right) .$

\begin{theorem}
Let $\mathbf{B}$ be a pointed category with binary products and coproducts,
kernels of split epis and coequalizers of reflexive graphs. The canonical
functor%
\begin{equation*}
Act\left( \mathbf{B}\right) \overset{T}{\longrightarrow }Pt\left( \mathbf{B}%
\right)
\end{equation*}%
is an equivalence if and only if the following two properties holds in $%
\mathbf{B}$:\newline
(A1) for every diagram of the form%
\begin{equation}
\xymatrix{
X \ar[r]^{\sigma \iota_1}  & F(X, B) \ar@<0.5ex>[r]^{\overline{[0 \, 1]}} & B \ar@<0.5ex>[l]^{\sigma \iota_2}
}%
\end{equation}%
the morphism $\sigma \iota _{1}$ is the kernel of $\overline{[0\ 1]}$;%
\newline
(A2) the split short five lemma holds.
\end{theorem}

\begin{proof}
Similar to Theorem \ref{Th 2}.
\end{proof}

\section{The general case}

Let%
\begin{equation*}
\mathbf{A}\overset{I}{\underset{G}{\underleftarrow{\overrightarrow{\ \ \ \ \
\ \ \ \ \ \ }}}}\mathbf{B\ \ ,\ \ }\pi :1\longrightarrow GI
\end{equation*}%
be a half-reflection.

Define a new category, denoted by $\mathbf{A}_{1}$\ as follows:\newline
Objects are pairs $\left( A,u\right) $ with $A\in \mathbf{A}$ and%
\begin{equation*}
u:A\longrightarrow GIA
\end{equation*}%
such that $I\left( u\right) =1.$\newline
A morphism $f:\left( A,u\right) \longrightarrow \left( A^{\prime },u^{\prime
}\right) $ is a morphism $f:A\longrightarrow A^{\prime }$ in $\mathbf{A}$
such that%
\begin{equation*}
\quadrado
{ A  }       { u  }     {  GIA   }
{ f  }                  {  GIf   }
{ A'  }       {  u' }     {  GIA'   }%
.
\end{equation*}

Define another category, denoted $\mathbf{A}_{2}$, as follows:\newline
Objects are systems 
\begin{equation*}
\left( \left( E,v\right) ,a,b,\left( A,u\right) \right)
\end{equation*}%
where $\left( E,v\right) $ and $\left( A,u\right) $ are objects in $\mathbf{A%
}_{1}$,%
\begin{equation*}
a:\left( E,v\right) \longrightarrow \left( A,u\right)
\end{equation*}%
is a morphism in $\mathbf{A}_{1}$, and 
\begin{equation*}
b:A\longrightarrow E
\end{equation*}%
is a morphism in $\mathbf{A}$ such that 
\begin{equation*}
ab=1_{A}.
\end{equation*}

Let $\mathbf{A}\overset{T}{\longrightarrow }$Pt$\left( \mathbf{B}\right) $
be any subcategory of Pt$\left( \mathbf{B}\right) $, not necessarily full,
we may consider the subcategories of reflexive graphs and internal
precategories in $\mathbf{B}$, restricted to split epis in $T\left( \mathbf{A%
}\right) $, and denote it respectively by $RG_{\mathbf{A}}\left( \mathbf{B}%
\right) $ and $PC_{\mathbf{A}}\left( \mathbf{B}\right) $.

In particular if the functor $G$, as above, admits a left adjoint $\left(
F,G,\eta ,\varepsilon \right) $ and $F$ is faithful and injective on
objects, then the canonical functor%
\begin{equation*}
\mathbf{A}\overset{T}{\longrightarrow }Pt\left( \mathbf{B}\right)
\end{equation*}%
determines a subcategory of split epis and so we have:\newline
Reflexive graphs internal to $\mathbf{B}$ and restricted to the split epis
in $T\left( \mathbf{A}\right) ,$ denoted $RG_{\mathbf{A}}\left( \mathbf{B}%
\right) $ as follows:%
\begin{equation*}
\xymatrix{
FA
\ar@<1ex>[rr]^{\varepsilon_{IA} F(\pi_A)}
\ar@<-1ex>[rr]_{c}
& & IA
\ar[ll]|{I(\eta_A)}
}%
\ \ \ \ ,\ \ \ cI\left( \eta _{A}\right) =1_{IA};
\end{equation*}%
for some $A\in \mathbf{A}$.\newline
Internal precategories in $\mathbf{B}$ relative to split epis in $T\left( 
\mathbf{A}\right) $, denoted by $PC_{\mathbf{A}}\left( \mathbf{B}\right) $,
as follows (where we use $\pi _{A}^{\prime }$ as an abbreviation to $%
\varepsilon _{IA}F\left( \pi _{A}\right) $ and similarly to $\pi
_{E}^{\prime }$)%
\begin{equation}
\xymatrix{
F(E)
\ar@<3ex>[rr]^{\pi'_E}
\ar@<-3ex>[rr]_{F(a)}
\ar[rr]|{m}
& & IE=FA
\ar@<-1.5ex>[ll]|{I(\eta_E)}
\ar@<1.5ex>[ll]|{F(b)}
\ar@<1.5ex>[rr]^{\pi'_A}
\ar@<-1.5ex>[rr]_{c}
& & IA
\ar[ll]|{I(\eta_A)}
}
\label{FE,IE=FA,IA}
\end{equation}%
for some%
\begin{equation*}
\xymatrix{
E
\ar@<1ex>[r]^{a}
& A
\ar[l]^{b}
}%
\ \ \ ,\ \ \ ab=1_{A}
\end{equation*}%
in $\mathbf{A}$, and satisfying the following conditions%
\begin{eqnarray}
cI\left( \eta _{A}\right) &=&1_{IA}  \label{c1} \\
I\left( a\right) &=&c  \label{c2} \\
I\left( b\right) &=&I\left( \eta _{A}\right)  \label{c3} \\
mI\left( \eta _{E}\right) &=&1_{IE}  \label{c4} \\
mF\left( b\right) &=&1_{IE}  \label{c5} \\
cm &=&cF\left( a\right)  \label{c6} \\
\pi _{A}^{\prime }m &=&\pi _{A}^{\prime }\pi _{E}^{\prime }.  \label{c7}
\end{eqnarray}

We observe that $c$ is determined by $a$, and $\left( \ref{c1}\right) $
follows from $\left( \ref{c2}\right) $, $\left( \ref{c3}\right) $ and the
fact that $ab=1_{A}$. We will be also interested in the notion of
multiplicative graph, which is obtained by removing $\left( \ref{c6}\right) $
and $\left( \ref{c7}\right) $ and in some cases we may be also interested in
removing $\left( \ref{c5}\right) $ so that the definition may be transported
from $\mathbf{B}$ to $\mathbf{A}$ and it does not depend on whether or not $%
G $ admits a left adjoint.

\begin{theorem}
For a half-reflection%
\begin{equation*}
\mathbf{A}\overset{I}{\underset{G}{\underleftarrow{\overrightarrow{\ \ \ \ \
\ \ \ \ \ \ }}}}\mathbf{B\ \ ,\ \ }\pi :1\longrightarrow GI,
\end{equation*}%
if the functor $G$ admits a left adjoint%
\begin{equation*}
\left( F,G,\eta ,\varepsilon \right) ,
\end{equation*}%
and $F$ is faithful and injective on objects, then%
\begin{eqnarray}
\mathbf{A}_{1} &\cong &RG_{\mathbf{A}}\left( \mathbf{B}\right)  \label{iso1}
\\
\mathbf{A}_{2}^{\ast } &\cong &PC_{\mathbf{A}}\left( \mathbf{B}\right)
\label{iso2}
\end{eqnarray}%
where $\mathbf{A}_{2}^{\ast }$ is the subcategory of $\mathbf{A}_{2}$ given
by the objects%
\begin{equation*}
\left( \left( E,v\right) ,a,b,\left( A,u\right) \right)
\end{equation*}%
such that%
\begin{eqnarray*}
IE &=&FA \\
I\left( a\right) &=&\varepsilon _{IA}F\left( u\right) \\
vb &=&\eta _{A} \\
G\left( \varepsilon _{IA}F\left( \pi _{A}\right) \right) \pi _{E} &=&G\left(
\varepsilon _{IA}F\left( \pi _{A}\right) \right) v.
\end{eqnarray*}
\end{theorem}

\begin{proof}
The isomorphism $\left( \ref{iso1}\right) $ is established by the adjuntion $%
\left( F,G,\eta ,\varepsilon \right) $. Given%
\begin{equation*}
A\overset{u}{\longrightarrow }GIA\ \ ,\ \ I\left( u\right) =1_{IA}
\end{equation*}%
we obtain%
\begin{equation}
\xymatrix{
FA
\ar@<1ex>[r]^{\pi'_A}
\ar@<-1ex>[r]_{u'}
&  IA
\ar[l]|{I(\eta_A)}
}
\label{FA,IA}
\end{equation}%
where $\pi _{A}^{\prime }=\varepsilon _{IA}F\left( \pi _{A}\right) $, $%
u^{\prime }=\varepsilon _{IA}F\left( u\right) $ and%
\begin{equation*}
u^{\prime }I\left( \eta _{A}\right) =1_{IA}\Leftrightarrow I\left( u\right)
=1.
\end{equation*}%
Conversely, given $\left( \ref{FA,IA}\right) $, we obtain $A$, since $F$ is
injective on objects, and 
\begin{equation*}
u=G\left( u^{\prime }\right) \eta _{A}.
\end{equation*}%
The isomorphism $\left( \ref{iso2}\right) $ is obtained as follows:\newline
Given $\left( \ref{FE,IE=FA,IA}\right) $, since $F$ in injective on objects
and faithful, we obtain%
\begin{equation*}
\xymatrix{
E
\ar@<1ex>[r]^{a}
& A
\ar[l]^{b}
}%
\ \ \ ,\ \ \ ab=1_{A}\ \ \ \ ,\ \ \ IE=FA.
\end{equation*}%
Now define on the one hand%
\begin{equation*}
u=G\left( c\right) \eta _{A}\ \ ,\ \ v=G\left( m\right) \eta _{E};
\end{equation*}%
while on the other hand%
\begin{equation*}
c=\varepsilon _{IA}F\left( u\right) \ \ ,\ \ m=\varepsilon _{IE}F\left(
v\right) ,
\end{equation*}%
and we have the following translation of equations%
\begin{equation*}
\begin{tabular}{|c|c|c|}
\hline
Eq. n.%
${{}^o}$
& in $\mathbf{B}$ & in $\mathbf{A}$ \\ \hline
\ref{c1} & $cI\left( \eta _{A}\right) =1_{IA}$ & $I\left( u\right) =1_{IA}$
\\ \hline
\ref{c4} & $mI\left( \eta _{E}\right) =1_{IE}$ & $I\left( v\right) =1_{IE}$
\\ \hline
\ref{c6} & $cm=cF\left( a\right) $ & $ua=GI\left( a\right) v$ \\ \hline
\ref{c2} & $I\left( a\right) =c$ & $I\left( a\right) =\varepsilon
_{IA}F\left( u\right) $ \\ \hline
\begin{tabular}{c}
\ref{c3} \\ 
\ref{c5}%
\end{tabular}
& 
\begin{tabular}{c}
$I\left( b\right) =I\left( \eta _{A}\right) $ \\ 
$mF\left( b\right) =1_{IE}$%
\end{tabular}
& $vb=\eta _{A}$ \\ \hline
\ref{c7} & $\pi _{A}^{\prime }m=\pi _{A}^{\prime }\pi _{E}^{\prime }$ & $%
G\left( \pi _{A}^{\prime }\right) \pi _{E}=G\left( \pi _{A}^{\prime }\right)
v$ \\ \hline
\end{tabular}%
\end{equation*}%
Note that $ua=GI\left( a\right) v$ follows from the fact that $a$ is a
morphism in $\mathbf{A}_{1}$, on the contrary of $b$ which is simply a
morphism in $\mathbf{A}$.
\end{proof}

In some cases we also have a functor%
\begin{equation*}
J:\mathbf{A}\longrightarrow \mathbf{B}
\end{equation*}%
satisfying the following three conditions:

\begin{enumerate}
\item $JG=1_{\mathbf{B}}$

\item the pair $\left( J\left( \eta _{A}\right) ,I\left( \eta _{A}\right)
\right) $ is jointly epic for every $A\in \mathbf{A}$, that is, given a pair
of morphisms $\left( f,g\right) $ as displayed below%
\begin{equation*}
\xymatrix{
JA
\ar[r]^{J(\eta_A)}
\ar[rd]_{f}
& FA
\ar@{-->}[d]^{[f \, g]}
& IA
\ar[l]_{I(\eta_A)}
\ar[ld]^{g}
\\
& B &
}%
\end{equation*}%
there is at most one morphism $\alpha :FA\longrightarrow B$, with the
property that $\alpha J\left( \eta _{A}\right) =f$ and $\alpha I\left( \eta
_{A}\right) =g$, denoted by $\alpha =[f,g]$ when it exists. Also the pair $%
\left( f,g\right) $ is said to be admissible (or cooperative in the sense of
Bourn and Gran \cite{Bourn&Gran}) w.r.t. $A$, if $[f,g]$ exists.

\item for every $A,E\in \mathbf{A}$, with $IE=FA$, a morphism, $u:J\left(
E\right) \longrightarrow FA$, such that $\left( u,1_{IE}\right) $ is
cooperative w.r.t. $E$ and satisfying $\pi _{A}[u\ 1]=\pi _{A}\pi _{E}$,
always factors trough $J\left( A\right) $, i.e., given $u$ as in the diagram
below%
\begin{equation*}
\xymatrix{
JE
\ar[r]^{J(\eta_E)}
\ar[rd]^{u}
\ar@{-->}[d]_{\bar{u}}
& FE
\ar@<0.5ex>[r]^{\pi'_E}
\ar@{..>}[d]^{[u \, 1]}
& IE
\ar@<0.5ex>[l]|{I(\eta_E)}
\ar@{=}[ld]
\\
JA
\ar[r]^{J(\eta_A)}
 & FA
\ar[r]^{\pi'_A}
 & IA
}%
\end{equation*}%
such that $[u\ 1]$ exists and $\pi _{A}[u\ 1]=\pi _{A}\pi _{E}$ then $%
u=J\left( \eta _{A}\right) u^{\prime }$ for a unique $u^{\prime
}:JE\longrightarrow JA$.
\end{enumerate}

\begin{theorem}
Let $\mathbf{B}$ be any category, with $\left( I,G,\pi \right) $%
\begin{equation*}
\mathbf{A}\overset{I}{\underset{G}{\underleftarrow{\overrightarrow{\ \ \ \ \
\ \ \ \ \ \ }}}}\mathbf{B\ \ ,\ \ }\pi :1\longrightarrow GI,
\end{equation*}%
a half-reflection such that the functor $G$ admits a left adjoint%
\begin{equation*}
\left( F,G,\eta ,\varepsilon \right) .
\end{equation*}%
If we can find a functor%
\begin{equation*}
J:\mathbf{A}\longrightarrow \mathbf{B}
\end{equation*}%
as above, then

\begin{itemize}
\item the category $RG_{\mathbf{A}}\left( \mathbf{B}\right) $ of reflexive
graphs in $\mathbf{B}$ relative to split epis from $\mathbf{A}$, is given by:%
\newline
Objects are pairs $\left( A,h\right) ,$ with$\ A\in \mathbf{A},$ and $%
h:JA\longrightarrow IA$ a morphism such that $\left( h,1_{IA}\right) $ is
cooperative w.r.t. $A$;\newline
A morphism $f:\left( A,h\right) \longrightarrow \left( A^{\prime },h^{\prime
}\right) $ is a morphism $f:A\longrightarrow A^{\prime }$ in $\mathbf{A}$
such that $h^{\prime }F\left( f\right) =I\left( f\right) h$.

\item the category of internal precategories in $\mathbf{B}$ relative to
split epis from $\mathbf{A}$, $PC_{\mathbf{A}}\left( \mathbf{B}\right) $, is
given by:\newline
Objects: 
\begin{equation*}
\left( A,E,a,b,t,h\right)
\end{equation*}%
where $A,E$, are objects in $\mathbf{A}$ , with $IE=FA$, $a,b,t,h,$ are
morphisms in $\mathbf{B}$,%
\begin{equation*}
\xymatrix{
JE
\ar@<1ex>[r]^{a}
\ar@<-1ex>[r]_{t}
&  JA
\ar[l]|{b}
\ar[r]^{h}
& IA
}%
\end{equation*}%
such that%
\begin{equation*}
ab=1=tb\ \ ,\ \ \ ha=ht
\end{equation*}%
and%
\begin{eqnarray*}
&&\left( h,1_{IA}\right) \ \ \text{and }\left( J\left( \eta _{E}\right)
b,I\left( \eta _{E}\right) I\left( \eta _{A}\right) \right) \ \ \ \text{are
cooperative w.r.t. }A \\
&&\left( J\left( \eta _{A}\right) a,I\left( \eta _{A}\right) [h\ 1]\right) \
\ \text{and }\left( J\left( \eta _{A}\right) t,1_{I\left( E\right) }\right)
\ \ \ \text{are cooperative w.r.t. }E
\end{eqnarray*}%
Morphisms are triples $\left( f_{3},f_{2},f_{1}\right) $ of morphisms%
\begin{equation*}
\xymatrix{
JE
\ar@<1ex>[r]^{a}
\ar@<-1ex>[r]_{t}
\ar[d]^{f_3}
&  JA
\ar[l]|{b}
\ar[r]^{h}
\ar[d]^{f_2}
& IA
\ar[d]^{f_1}
\\
JE'
\ar@<1ex>[r]^{a'}
\ar@<-1ex>[r]_{t'}
&  JA'
\ar[l]|{b'}
\ar[r]^{h'}
& IA'
}%
\end{equation*}%
such that the obvious squares in the above diagram commute and furthermore
the pair $\left( J\left( \eta _{A^{\prime }}\right) f_{2},I\left( \eta
_{A^{\prime }}\right) f_{1}\right) $ is admissible w.r.t. $A$ while the pair 
$\left( J\left( \eta _{E^{\prime }}\right) f_{3},I\left( \eta _{E^{\prime
}}\right) [J\left( \eta _{A^{\prime }}\right) f_{2}\ \ I\left( \eta
_{A^{\prime }}\right) f_{1}]\right) $ is admissible w.r.t. $E$.
\end{itemize}
\end{theorem}

\begin{proof}
Calculations are similar to the previous sections and the resulting diagram $%
\left( \ref{FE,IE=FA,IA}\right) $ is given by%
\begin{eqnarray*}
c &=&[h\ 1] \\
m &=&[J\left( \eta _{A}\right) t\ \ 1_{I\left( E\right) }] \\
F\left( b\right) &=&[J\left( \eta _{E}\right) b\ \ I\left( \eta _{E}\right)
I\left( \eta _{A}\right) ] \\
F\left( a\right) &=&[J\left( \eta _{A}\right) a\ \ I\left( \eta _{A}\right)
[h\ 1]].
\end{eqnarray*}%
The same argument applies to obtain the morphisms.
\end{proof}

\subsection{The example of unitary magmas with right cancellation}

An example of a general situation in the conditions of the above theorem is
the following one.

Let $\mathbf{B}$ be a pointed category with kernels of split epis, with
binary products and coproducts and such that the pair $\left( \left\langle
1,0\right\rangle ,\left\langle 1,1\right\rangle \right) $, as displayed%
\begin{equation*}
\xymatrix{
B
\ar[r]^{ <1,0> }
& B \times B
\ar@<0.5ex>[r]^{\pi'_2}
& B
\ar@<0.5ex>[l]^{ <1,1> }
}%
\end{equation*}%
is jointly epic for every $B\in \mathbf{B}$, and then consider: $\mathbf{A}$%
, the full subcategory of Pt$\left( \mathbf{B}\right) $ given by the split
epis with the property that 
\begin{equation*}
\xymatrix{
X
\ar[r]^{ \ker \alpha}
& A
\ar@<0.5ex>[r]^{\alpha}
& B
\ar@<0.5ex>[l]^{ \beta }
}%
\end{equation*}%
the pair $\left( \ker \alpha ,\beta \right) $ is jointly epic (identifying $%
\left( A,\alpha ,\beta ,B\right) $ with $\left( A^{\prime },\alpha ^{\prime
},\beta ^{\prime },B\right) $ whenever $A\cong A^{\prime }$, in order to
obtain $F$ injective on objects).

Then we have functors%
\begin{eqnarray*}
I,F,J &:&\mathbf{A}\longrightarrow \mathbf{B} \\
G &:&\mathbf{B}\longrightarrow \mathbf{A}
\end{eqnarray*}%
with%
\begin{eqnarray*}
I\left( A,\alpha ,\beta ,B\right) &=&B \\
F\left( A,\alpha ,\beta ,B\right) &=&A \\
J\left( A,\alpha ,\beta ,B\right) &=&X\ \ ,\text{the object kernel of }\alpha
\\
G\left( B\right) &=&\left( B\times B,\pi _{2},\left\langle 1,1\right\rangle
,B\right)
\end{eqnarray*}%
and with $\pi :1_{\mathbf{A}}\longrightarrow GI$ given by $\pi =[0\ 1]$.

An example of such a category is the category of unitary magmas with right
cancellation. Also every strongly unital category satisfies the above
requirements (see \cite{BB}, and references there).

\end{document}